\documentclass[11pt]{amsart}
\usepackage{amsmath}
\usepackage{amsthm}
\usepackage{amssymb}
\usepackage{amscd}
\usepackage{enumerate}

\theoremstyle{plain}
\newtheorem{theorem}{Theorem}
\newtheorem{lemma}[theorem]{Lemma}

\newtheorem{proposition}[theorem]{Proposition}
\newtheorem*{claim}{Claim}
\newtheorem*{case1}{Case 1}
\newtheorem*{case2}{Case 2}

\theoremstyle{definition}
\newtheorem{definition}[theorem]{Definition}
\newtheorem{setting}[theorem]{Setting}

\numberwithin{theorem}{section}

\def\forces{\mathrel {||}\joinrel \relbar}

\DeclareMathOperator{\cf}{cf}
\DeclareMathOperator{\cof}{cof}
\DeclareMathOperator{\Coll}{Coll}

\DeclareMathOperator{\dom}{dom}
\DeclareMathOperator{\Fil}{Fil}

\DeclareMathOperator{\lh}{lh}

\DeclareMathOperator{\ssup}{ssup}

\DeclareMathOperator{\Ult}{Ult}

\begin{document}

	\title{Universal graphs at $\aleph_{\omega_1+1}$}
	\author{Jacob Davis}
	\address{Department of Mathematical Sciences, Carnegie Mellon University, 5000 Forbes Avenue, Pittsburgh, PA 15213}

	\begin{abstract}
		Starting from a supercompact cardinal we build a model in which $2^{\aleph_{\omega_1}}=2^{\aleph_{\omega_1+1}}=\aleph_{\omega_1+3}$ but there is a jointly universal family of size $\aleph_{\omega_1+2}$ of graphs on $\aleph_{\omega_1+1}$. The same technique will work for any uncountable cardinal in place of $\omega_1$.
	\end{abstract}

	\maketitle

\section{Introduction}

For a cardinal $\mu$, a {\em universal graph on $\mu$} is a graph on $\mu$ into which every graph on $\mu$ can be embedded as an induced subgraph. A family of graphs on $\mu$ is {\em jointly universal on $\mu$} if every graph on $\mu$ can be embedded into at least one of them. We are interested in obtaining jointly universal families of small cardinality for $\mu$ a successor cardinal of the form $\kappa^+$.

If $2^{\kappa}=\kappa^+$ then by a standard model-theoretic construction there is a saturated and hence universal graph on $\kappa^+$. This holds even if $2^{\kappa^+}$ is large. So we are interested in cases when $2^{\kappa}>\kappa^+$. If $\kappa$ is regular then as shown by D{\v z}amonja and Shelah in \cite{universalModels} it is consistent to have a jointly universal family on $\kappa^+$ of size $\kappa^{++}$ whilst $2^{\kappa^+}$ is arbitrarily large. If $\kappa$ is singular than matters are generally more problematic. D{\v z}amonja and Shelah introduce a new approach in \cite{2author} that begins with $\kappa$ supercompact and performs a preparatory iteration to add functions that after Prikry forcing will become embeddings into a family of  jointly universal graphs, whilst preserving some of the supercompactness of $\kappa$, followed by Prikry forcing to change the cofinality of $\kappa$. This enables them to build a model where $\cf(\kappa)=\omega$, $2^{\kappa^+}>\kappa^{++}$ and there is a jointly universal family on $\kappa^{+}$ of size $\kappa^{++}$. In \cite{5author} Cummings, D{\v z}amonja, Magidor, Morgan and Shelah modify this construction to use Radin forcing and achieve $\cf(\kappa)>\omega$ and $2^{\kappa^+}>\kappa^{++}$ with a jointly universal family on $\kappa^+$ of size $\kappa^{++}$. Then in \cite{3author} Cummings, D{\v z}amonja and Morgan employ Prikry forcing with interleaved collapses to build a model with $2^{\aleph_{\omega+1}}>\aleph_{\omega+2}$ and a jointly universal family on $\aleph_{\omega+1}$ of size $\aleph_{\omega+2}$. We will use a preparatory forcing followed by Radin forcing with interleaved collapses to prove the following theorem.

\begin{theorem}
	Let $\kappa$ be supercompact and $\lambda<\kappa$ regular uncountable. Then there is a forcing extension in which $\kappa=\aleph_{\lambda}$, $2^{\aleph_{\lambda}}=2^{\aleph_{\lambda+1}}=\aleph_{\lambda+3}$ and there is a jointly universal family of graphs on $\aleph_{\lambda+1}$ of size $\aleph_{\lambda+2}$.
\end{theorem}

In section \ref{ultrafilter_sequences} we consider sequences of ultrafilters $\vec{u}$ from which it is possible to derive a version $\mathbb{R}_{\vec{u}}$ of Radin forcing with interleaved collapses. The forcing is similar to the one used by Foreman and Woodin in \cite{ForemanWoodin} but differs in the forcing interleaved and some technical details. Also we will show the desired properties of the forcing directly rather than proving that a supercompact Radin forcing has these properties and then projecting them.

We identify certain useful properties of sequences of ultrafilters that have been derived from supercompactness embeddings, and denote the class of sequences possessing these properties by $\mathcal{U}$. In section \ref{properties_of_radin} we prove some results about the forcing $\mathbb{R}_{\vec{u}}$ when $\vec{u}\in\mathcal{U}$; in particular that it has the Prikry property and that its generic filters can be conveniently characterised. In section \ref{preparatory_forcing} we define a preparatory forcing $\mathbb{Q}_{\vec{u}}$ that adds functions which, after Radin forcing, will become embeddings from graphs on $\kappa^+$ into a graph on $\kappa^+$ that we intend to make a member of our jointly universal family. We also prove that this preparatory forcing has properties including $\kappa$-directed closure and the $\kappa^+$-cc.

In section \ref{construction_of_model} we begin with $\kappa$ supercompact and perform a Laver preparation forcing. We then use a diamond sequence to identify ultrafilter sequences $\vec{u}^{\gamma}$ for $\gamma<\kappa^{+4}$, and carry out an iteration of the $\mathbb{Q}_{\vec{u}^{\gamma}}$ forcings. This allows us to extend a supercompactness embedding $j$ from $V$ to the generic extension, and from this $j$ we derive an ultrafilter sequence $\vec{u}$ in $\mathcal{U}$ and take $J$ that is $\mathbb{R}_{\vec{u}}$-generic over the universe resulting from the $\mathbb{Q}_{\vec{u}^{\gamma}}$-iteration. We show that there is a stationary set $S$ of points $\gamma$ in $\kappa^{+4}$ where $\vec{u}$ restricts to $\vec{u}^{\gamma}$ and $\vec{u}^{\gamma}\in\mathcal{U}$; then the characterisation of generic filters will show that $J$ is also generic for $\mathbb{R}_{\vec{u}^{\gamma}}$ over the same universe. Our final model will be built by stopping the iteration at a point in $S$ that is also a limit of $\kappa^{++}$-many members of $S$ and adjoining $J$; we will then have constructed $\kappa^{++}$-many graphs to use as members of our joint universal family, together with embeddings of every graph on $\kappa^+$ into them.

We will write $x:=y$ to mean $x$ is defined to equal $y$, and $x=:y$ to mean $y$ is defined to equal $x$. We write $f:A\rightharpoonup B$ for a partial function from $A$ to $B$. For forcing conditions $p$ and $q$ we write $p\parallel q$ to mean $p$ is compatible with $q$; for a formula $\varphi$ we write $p\parallel\varphi$ to mean $p$ decides whether or not $\varphi$ is true. Given an ultrafilter $u$, the quantification $\forall_u x:\varphi(x)$ will signify that $\{x\mid\varphi(x)\}\in u$. Our forcing convention is that $p\leq q$ means $p$ is stronger (more informative) than $q$.

\section{Ultrafilter sequences and the definition of $\mathbb{R}_{\vec{u}}$} \label{ultrafilter_sequences}

\subsection{The nature of ultrafilter sequences} We will be building sequences of the following form.

\begin{definition}
	A sequence $\vec{u}=\langle \kappa, u_i, \mathcal{F}_i\mid i<\lambda \rangle$ (which means that there is a single $\kappa$ together with $\lambda$-many each of the $u_i$ and $\mathcal{F}_i$) is a {\em proto ultrafilter sequence} if $\lambda<\kappa$, the $u_i$ are $\kappa$-complete ultrafilters on $V_{\kappa}$ and the $\mathcal{F}_i$ are sets of partial functions from $V_{\kappa}$ to $V_{\kappa}$. We will write $\kappa(\vec{u})$ for $\kappa$ and $\lh\vec{u}$ for $\lambda$ which we also call the {\em length} of $\vec{u}$. We stress that our use of the term ``length'' here differs from the usual convention.
\end{definition}

For $\beta$ a strongly inaccessible cardinal we define $\mathbb{C}(\alpha,\beta)$ to be the poset $\Coll(\alpha^{+5},<\beta)$ and $\mathbb{B}(\alpha,\beta)$ to be the regular open algebra derived from this poset. Note that $\mathbb{C}(\alpha,\beta)$ is contained in $V_{\beta}$ and has the $\beta$-cc so we are free to regard conditions in $\mathbb{B}(\alpha,\beta)$ as members of $V_{\beta}$. Given sequences $\vec{v}$ and $\vec{w}$ with $\kappa(\vec{v})<\kappa(\vec{w})$ we will also write $\mathbb{B}(\vec{v},\vec{w})$ for $\mathbb{B}(\kappa(\vec{v}),\kappa(\vec{w}))$. This is the forcing that we will interleave into our Radin generic sequence.

\begin{definition}
	Let $\kappa$ be strongly inaccessible, $i<\kappa$ and $u$ a $\kappa$-complete ultrafilter on $V_{\kappa}$ concentrating on proto ultrafilter sequences of length $i$. Then a {\em $u$-constraint} is a partial function $h:V_{\kappa}\rightharpoonup V_{\kappa}$ such that:
	\begin{itemize}
		\item $\dom h$ is in $u$ and consists of proto ultrafilter sequences of length $i$.
		\item For all $\vec{w}$ in $\dom h$, $h(\vec{w})\in \mathbb{B}(\kappa(\vec{w}),\kappa)-\{0\}$.
	\end{itemize}
	An {\em ultrafilter sequence} is defined by recursion on $\kappa(\vec{u})$ to be a proto ultrafilter sequence $\vec{u}=\langle \kappa, u_i, \mathcal{F}_i\mid i<\lambda\rangle$ such that each $\mathcal{F}_i$ is a non-empty set of $u_i$-constraints, and each $u_i$ concentrates on ultrafilter sequences of length $i$.
\end{definition}

Observe that if we form the ultrapower $j_u:V\rightarrow\Ult(V,u)$ we can regard $u$-constraints (modulo $u$) as representing members of the Boolean algebra $\mathbb{B}(\kappa,j_u(\kappa))^{\Ult(V,u)}$.

\begin{definition}
	We will need an auxiliary notion of {\em supercompact ultrafilter sequences}. Such sequences will be recursively defined to have the form $\vec{u}^*=\langle z, u^*_i, H^*_i \mid i<\lambda\rangle$ where there is some $\kappa(\vec{u}^*):=\kappa>\lambda$ with $z$ a set of ordinals that is a superset of $\kappa$, each $u^*_i$ is an ultrafilter on $[\kappa^{+4}]^{<\kappa}\times V_{\kappa}^{2i}$ that concentrates on supercompact ultrafilter sequences of length $i$, and each $H^*_i$ is a {\em $u^*_i$-constraint}. This last means that $\dom H^*_i \in u^*_i$, and for $\vec{w}^*\in\dom H^*_i$ we have $H^*_i(\vec{w}^*)\in\mathbb{B}(\vec{w}^*,\vec{u}^*)-\{0\}$.

	We also define an ordering on $u^*_i$-constraints by $L^*\leq K^*$ if $\dom L^*\subseteq\dom K^*$ and $L^*(\vec{w}^*)\leq K^*(\vec{w}^*)$ for all $\vec{w}^*$ in $\dom L^*$. We shall use similar orderings for other functions whose domains are required to lie in some ultrafilter.
\end{definition}

Observe that if we form the ultrapower $j_{u^*}:V\rightarrow\Ult(V,u^*)$ then we can regard $u^*$-constraints (modulo $u^*$) as representing members of the Boolean algebra $\mathbb{B}(\kappa, j_{u^*}(\kappa))^{\Ult(V,u^*)}$.

\subsection{Constructing ultrafilter sequences}

For the remainder of this section we work in the following context.

\begin{setting}
	Let $2^{\kappa} = \kappa^{+4}$ with $j:V\rightarrow M$ witnessing that $\kappa$ is $\kappa^{+4}$-supercompact. Let $\lambda<\kappa$ be regular uncountable.
\end{setting}

We will use $j$ to inductively build an ultrafilter sequence $\vec{u}=\langle \kappa, u_i, \mathcal{F}_i\mid i<\lambda \rangle$ with $\kappa(\vec{u})=\kappa$. In doing so we will need to construct an auxiliary supercompact ultrafilter sequence $\vec{u}^*=\langle j``\kappa^{+4}, u^*_i, H^*_i\mid i<\lambda \rangle$.

We will also define a function $\pi$ from supercompact ultrafilter sequences to ultrafilter sequences, given by
$$\pi(\langle z^*, w^*_i, K^*_i \mid i<\bar{\lambda}\rangle) := \langle z^*\cap\bar{\kappa}, \pi_i(w^*_i), \pi'_i(w^*_i,K^*_i)\mid i<\bar{\lambda}\rangle$$
with $\pi_i$ and $\pi'_i$ to be built as part of the induction and $\bar{\kappa}:=\kappa(\langle z^*, w^*_i, K^*_i \mid i<\bar{\lambda}\rangle)$. Note that the $u^*_i$ concentrate on sequences where $z^*\cap\bar{\kappa}$ is inaccessible. We will ensure as we induct on $\bar{\lambda}\leq\lambda$ that
\begin{equation}\tag{*}
	j(\pi)(\langle j` `\kappa^{+4}, u^*_i, H^*_i \mid i<\bar{\lambda}\rangle) = \langle \kappa, u_i, \mathcal{F}_i \mid i<\bar{\lambda}\rangle.
\end{equation}
Suppose we have defined $u_i$, $u^*_i$, $\mathcal{F}_i$, $H^*_i$, $\pi_i$ and $\pi'_i$ for $i<\bar{\lambda}$; this gives us the definition of $\pi$ on sequences of length up to $\bar{\lambda}$. Define
$$u^*_{\bar{\lambda}} := \{X\subseteq[\kappa^{+4}]^{<\kappa}\times V_{\kappa}^{2\bar{\lambda}} \mid \langle j``\kappa^{+4}, u^*_i, H^*_i \mid i<\bar{\lambda}\rangle \in j(X)\}$$
and
$$u_{\bar{\lambda}} := \{Y\subseteq V_{\kappa}\mid \langle \kappa, u^i, \mathcal{F}^i \mid i<\bar{\lambda}\rangle \in j(Y)\}.$$
For $w^*$ an ultrafilter on $[\bar{\kappa}^{+4}]^{<\bar{\kappa}}\times V_{\bar{\kappa}}^{2\bar{\lambda}}$ define
$$\pi_{\bar{\lambda}}(w^*):= \{Y\subseteq V_{\bar{\kappa}}\mid \pi^{-1}``Y \in w^* \}.$$
Note by (*) that $Y\in u_{\bar{\lambda}}$ is equivalent to $\pi^{-1}``Y \in u^*_{\bar{\lambda}}$ which, since $\pi$ and $j(\pi)$ agree on $V_{\kappa}$, is equivalent to $j(\pi)^{-1}``Y\in u^*_{\bar{\lambda}}$ and so to $Y\in j(\pi_{\bar{\lambda}})(u^*_{\bar{\lambda}})$. Therefore $u_{\bar{\lambda}}=j(\pi_{\bar{\lambda}})(u^*_{\bar{\lambda}})$. We now pause the construction to make some definitions.

\begin{definition}
	Let $w^*$ be an ultrafilter on $[\bar{\kappa}^{+4}]^{<\bar{\kappa}}\times V_{\bar{\kappa}}^{2\bar{\lambda}}$ that concentrates on supercompact ultrafilter sequences of length $\bar{\lambda}$, and $K^*$ a $w^*$-constraint. Then for $A\in w^*$ and $\vec{x}$ an ultrafilter sequence we define
	$$b(K^*, A)(\vec{x}):=\bigvee\{K^*(\vec{x}^*)\mid \pi(\vec{x}^*)=\vec{x}, \vec{x}^*\in A\}\in\mathbb{B}(\kappa(\vec{x}),\bar{\kappa})-\{0\}.$$
	Observe that the domain of $b(K^*,A)$ is the projection of $A$ under $\pi_{\bar{\lambda}}$, so it is in $\pi_{\bar{\lambda}}(w^*)$. Observe also that for $A'\subseteq A$ we have $b(K^*,A')\leq b(K^*,A)$ pointwise, so as $A$ ranges over $w^*$ the equivalence classes generated by the $b(K^*,A)$ yield a non-trivial filter base in $\mathbb{B}(\bar{\kappa},j_{\pi_{\bar{\lambda}}(w^*)}(\bar{\kappa}))^{\Ult(V,\pi_{\bar{\lambda}}(w^*))}$. We shall call the induced filter $\Fil(K^*)$.
\end{definition}

Now given $w^*$ and $K^*$ we define $\pi'_{\bar{\lambda}}(w^*,K^*) = \{g\mid [g]_{\pi_{\bar{\lambda}}(w^*)}\in \Fil(K^*)\}$, which will conclude our definition of $\pi$ for sequences of length up to $k+1$. We note that $\pi'_{\bar{\lambda}}(w^*,K^*)$ consists of all $g$ such that $g\geq b(K^*, A)$ for some $A\in w^*$, where we ensure a pointwise inequality be shrinking the $A$ as necessary. It remains to choose $H^*_{\bar{\lambda}}$, and then once we have done so we will conclude by defining $\mathcal{F}_{\bar{\lambda}} = j(\pi'_{\bar{\lambda}})(u^*_{\bar{\lambda}}, H^*_{\bar{\lambda}})$, which is to say $\mathcal{F}_{\bar{\lambda}}=\{h\mid [h]_{u_{\bar{\lambda}}}\in\Fil(H^*_{\bar{\lambda}})\}$. Care must be taken in selecting $H^*_{\bar{\lambda}}$ because we wish to ensure that the filter $\mathcal{F}_{\bar{\lambda}}$ it induces will be an ultrafilter. The following lemma will be helpful to that end.

\begin{lemma}
	Let $b\in \mathbb{B}(\kappa,j_{_{\bar{\lambda}}}(\kappa))^{\Ult(V,_{\bar{\lambda}})}$ and $K^*$ a $u^*_{\bar{\lambda}}$-constraint. Then there is a $u^*_{\bar{\lambda}}$-constraint $L^*\leq K^*$ such that either $b\in\Fil(L^*)$ or $\neg b\in\Fil(L^*)$.
\end{lemma}

\begin{proof}
	Say $b =: [f]_{_{\bar{\lambda}}}$ and define $A:=\{\vec{x}^*\in \dom K^*\mid \pi(\vec{x})\in \dom f\} \in u^*_{\bar{\lambda}}$. Then for each $\vec{x}^* \in A$ take $L^*(\vec{x}^*)\leq K(\vec{x}^*)$ such that either $L^*(\vec{x}^*)\leq f(\pi(\vec{x}^*))$ or $L^*(\vec{x}^*)\leq \neg f(\pi(\vec{x}^*))$. Define $A^+$ to be the set of places in $A$ where the first case occurs, and $A^-$ to be the set of places where the second does. One of these is in $u^*_{\bar{\lambda}}$ and restricting the domain of $L^*$ to this set will give $L^*$ the required properties.
\end{proof}

For $\bar{\kappa}<\kappa$ the forcing $\mathbb{C}(\bar{\kappa}, \kappa)$ has the $\kappa$-chain condition, so $|\mathbb{B}(\bar{\kappa},\kappa)|=\kappa$. This tells us by elementarity that $|\mathbb{B}(\kappa,j_{_{\bar{\lambda}}}(\kappa))^{\Ult(V,_{\bar{\lambda}})}| = |j_{_{\bar{\lambda}}}(\kappa)| = 2^{\kappa} = \kappa^{+4}$. Now the $u^*_{\bar{\lambda}}$-constraints can be regarded as members of the regular open algebra $\mathbb{B}(\kappa, j_{u^*_{\bar{\lambda}}}(\kappa))^{\Ult(V,u^*_{\bar{\lambda}})}$, in the non-zero part of which the forcing $\mathbb{C}(\kappa, j_{u^*_{\bar{\lambda}}}(\kappa))^{\Ult(V,u^*_{\bar{\lambda}})}=\Coll(\kappa^{+5},<j_{u^*_{\bar{\lambda}}}(\kappa))^{\Ult(V,u^*_{\bar{\lambda}})}$ is dense. The $\kappa^{+4}$-supercompactness of $j_{u^*_{\bar{\lambda}}}$ tells us that the latter forcing is $\kappa^{+5}$-closed, so we can repeatedly apply the above lemma to obtain a $u^*_{\bar{\lambda}}$-constraint $H^*_{\bar{\lambda}}$ such that $\Fil(H^*_{\bar{\lambda}})$ is an ultrafilter. This concludes the inductive construction.

\subsection{Properties of ultrafilter sequences I}

We collect together all save one of the properties that we will want our ultrafilter sequences to possess. The final property is postponed because it requires the definition of $\mathbb{R}_{\vec{u}}$ to state.

Note that for $h\in \mathcal{F}_i$ and $s\in V_{\kappa}$ the $h\downharpoonright s$ notation used here means that the domain of $h$ is restricted to $\{\vec{w}\mid s \in V_{\kappa(\vec{w})}\}$.

\begin{definition}
	We define $\mathcal{U}'$ to be the class of all ultrafilter sequences $\vec{u}=\langle \bar{\kappa}, u_i, \mathcal{F}_i \mid i < \bar{\lambda} \rangle$ that satisfy the following properties:
	\begin{enumerate}
	 	\item The ultrafilter $u_i$ concentrates on sequences from $\mathcal{U}'$ of length $i$ (so this definition is recursive).
	 	\item If $h$ is in $\mathcal{F}_i$ and $h$ is equal to $h'$ modulo $u_i$ then $h'$ is also in $\mathcal{F}_i$.
	 	\item The set of Boolean values represented by the functions $\mathcal{F}_i$ is a $\bar{\kappa}$-complete ultrafilter on $\mathbb{B}(\bar{\kappa},j_{u_i}(\bar{\kappa}))^{\Ult(V,u_i)}$.
		\item (Normality) For all $i<\bar{\lambda}$, given $\langle h^s \mid s \in V_{\bar{\kappa}} \rangle$ with $h^s \in \mathcal{F}_i$ then there is $h \in \mathcal{F}_i$ such that $h \leq h^s \downharpoonright s$ for all $s \in V_{\bar{\kappa}}$.
		\item Let $i'<i''<\bar{\lambda}$ and $e\in \mathcal{F}_{i'}$. Then there is a $u_{i''}$-large set of ultrafilter sequences $\vec{w}=\langle\bar{\kappa}(\vec{w}), w_i, \mathcal{G}_i \mid i<i''\rangle$ such that $e\upharpoonright\bar{\kappa}(\vec{w})$ is in $\mathcal{G}_{i'}$.
	\end{enumerate}
\end{definition}

\begin{lemma}
	Let $\vec{u}$ be constructed from $j$ as above. Then $\vec{u}\in\mathcal{U'}$.
\end{lemma}

\begin{proof}
	The first three clauses are immediate.
	\begin{enumerate} [(1)]
		\setcounter{enumi}{3}
		\item (Normality) We are given $i<\lambda$ and $\langle h^s \mid s \in V_{\kappa} \rangle \subseteq \mathcal{F}_i$. Say $h^s\geq b(H^*_i,A^s)$ with $A^s\in u^*_i$. Take the diagonal intersection of the $A^s$,
	$$A:=\{\vec{w}^* \mid \forall s \in V_{\kappa(\vec{w}^*)}: \vec{w}^*\in A^s\}.$$
	We have $\forall s\in V_{\kappa(\vec{u}^*\upharpoonright i)}: \vec{u}^*\upharpoonright i \in j(A^s)$, which is to say $\vec{u}^*\upharpoonright i\in j(A)$ so $A\in u^*_i$. Then $h:=b(H^*_i,A)$ will be our candidate.

	Given $s\in V_{\kappa}$ we want $h \leq h^s \downharpoonright s$, so given $\vec{w}\in \dom h$ above $s$ we want $h(\vec{w})\leq h^s(\vec{w})$. Now $h(\vec{w})$ is the supremum of $K^*(\vec{w}^*)$ over $\vec{w}^*\in A$ such that $\pi(\vec{w}^*)=\vec{w}$. All of these $\vec{w}^*$ have $\kappa(\vec{w}^*)=\kappa(\vec{w})$ above $s$, so they must also be members of $A^s$. But $h^s(\vec{w})$ is the supremum of $K^*(\vec{w}^*)$ over members of $A^s$, so $h^s(\vec{w})\geq h(\vec{w})$.
		\item We are given $i'<i''<\lambda$ and $e\in \mathcal{F}_{i'}$ and note that $j(e)\upharpoonright\kappa(\vec{u}\upharpoonright i'') = e \in \mathcal{F}_{i'}$. Then by elementarity there is a $u_{i''}$-large set of sequences $\vec{w}=\langle \kappa(\vec{w}), w_i, \mathcal{G}_{i'}\rangle$, as required.
	\end{enumerate}
\end{proof}

\subsection{Definition of the Radin forcing $\mathbb{R}_{\vec{u}}$} \label{defineR}

We are given an ultrafilter sequence $\vec{u}\in\mathcal{U'}$ and define $\bar{\kappa}:=\kappa(\vec{u})$ and $\bar{\lambda}:=\lh\vec{u}$.

For notational convenience, given $\vec{w} =: \langle \kappa(\vec{w}), w_i, \mathcal{F}_i \mid i<\lh\vec{w}\rangle$ we will start writing $\mathcal{F}_{\vec{w},i}$ for $\mathcal{F}_i$ and $\mathcal{F}_{\vec{w}}$ for the set of
functions $e:V_{\kappa(\vec{w})}\rightharpoonup V_{\kappa(\vec{w})}$ such that defining $e_i:=e\upharpoonright\{\vec{v}\mid\lh\vec{v}=i\}$ gives us $e_i\in \mathcal{F}_{\vec{w},i}$ for all $i<\lh\vec{w}$. Note that $\dom e$ is permitted to include sequences that are longer than $\vec{w}$ itself.

\begin{definition}
	Let $\vec{w}\in\mathcal{U}'$. Then an {\em upper part} for $\mathbb{R}_{\vec{w}}$ is a member $e$ of $\mathcal{F}_{\vec{w}}$ such that:
	$$\forall\vec{v}\in\dom e: e\upharpoonright\kappa(\vec{v})\in \mathcal{F}_{\vec{v}}.$$
\end{definition}

We claim that any $e$ in $\mathcal{F}_{\vec{w}}$ can have its domain shrunk to make it into an upper part. Define $e^0:=e$ and then by the final clause of the definition of $\mathcal{U}'$ we have that
$$A^1:=\{\vec{v}\in\dom e\mid e^0\upharpoonright\kappa(\vec{v})\in \mathcal{F}_{\vec{v}}\}\in \bigcap w_i$$
so we can define $e^1:=e^0\upharpoonright A^1 \in \mathcal{F}_{\vec{w}}$. Iterating this process $\omega$ times and intersecting the $A^n$ we reach $e'\leq e$ which has the required property. From now on we shall perform such shrinking without comment when building forcing conditions.

\begin{definition}
	A {\em suitable triple} is $(\vec{w},e,q)$ satisfying the following conditions:
	\begin{itemize}
		\item $\vec{w}\in\mathcal{U}'$.
		\item  $e$ is an upper part for $\mathbb{R}_{\vec{w}}$.
		\item $q \in \mathbb{B}(\kappa(\vec{w}),\kappa)-\{0\}$.
	\end{itemize}
	A {\em direct extension} of $(\vec{w},e,q)$ is a suitable triple $(\vec{w},e',q')$ such that:
	\begin{itemize}
		\item $e'\leq e$ (i.e. $\dom e'\subseteq\dom e$ and $e' \leq e$ pointwise).
		\item $q' \leq q$.
	\end{itemize}
	Another suitable triple $(\vec{v},d,p)$ is {\em addable below} $(\vec{w},e,q)$ if it satisfies the following:
	\begin{itemize}
		\item $\vec{v} \in \dom e$.
		\item $d \leq e\upharpoonright\kappa(\vec{v})$.
		\item $p\leq e(\vec{v})$.
	\end{itemize}
\end{definition}

We observe that for every $\vec{v}\in\dom e$ the definition of ``upper part'' has assured us that $(\vec{v},e\upharpoonright\kappa(\vec{v}),e(\vec{v}))$ is both a suitable triple and addable below $(\vec{w},e,q)$.

\begin{definition}
	A {\em condition} in $\mathbb{R}_{\vec{u}}$ is a finite sequence
	$$s = ((\vec{w}_0,e_0,q_0),...,(\vec{w}_{n-1},e_{n-1},q_{n-1}), (\vec{u}, h))$$
	such that each $(\vec{w}_k,e_k,q_k)$ is a suitable triple, the $\kappa(\vec{w}_k)$ are increasing, $q_k\in\mathbb{B}(\vec{w}_k,\vec{w}_{k+1})$, and $h$ is an upper part for $\mathbb{R}_{\vec{u}}$. We also require that $\kappa(\vec{w}_0)=\omega$, $\lh\vec{w}_0=0$ and $e_0=\phi$. We will call such a $((\vec{w}_0,e_0,q_0),...,(\vec{w}_{n-1},e_{n-1},q_{n-1}))$ a {\em lower part} for the forcing.

	{\em Extension} in $\mathbb{R}_{\vec{u}}$ is given by $s'\leq s$ if
	$$s'=((\vec{v}_0,d_0,p_0),...,(\vec{v}_{m-1},d_{m-1},p_{m-1}),(\vec{u},h'))$$
	such that $h'\leq h$, every $\vec{w}_k$ occurs as some $\vec{v}_l$, and every $(\vec{v}_l,d_l,p_l)$ is either a direct extension of one of the $(\vec{w}_k,e_k,q_k)$ or addable below one of them or addable below $(\vec{u},h)$.

	{\em Direct extension} in $\mathbb{R}_{\vec{u}}$ is given by $s'\leq^* s$ if
	$$s'=((\vec{w}_0,e'_0,q'_0),...,(\vec{w}_{n-1},e'_{n-1},q'_{n-1}), (\vec{u}, h'))$$
	with $h'\leq h$ and $(\vec{w}_k,e'_k,q'_k)$ a direct extension of $(\vec{w}_k,e_k,q_k)$ for $k<n$. \

	For lower parts $r$ and $r'$ we define extension $r'\leq r$ in the same way as for conditions, except that all triples from $r'$ must by either direct extensions of, or addable below, a triple from $r$. Note that this compels $\kappa(\max r') = \kappa(\max r)$. Likewise we have a notion of $\leq^*$ on lower parts, and a {\em $^*$-open} set of lower parts is one that is downward-closed under this relation.
\end{definition}

Observe that a forcing condition is required to have a triple $((\langle\omega\rangle,\phi,p))$ as a member of its stem for some $p$. However we shall write $((\vec{u},h))$ as an abbreviation for $((\langle\omega\rangle,\phi,\phi),(\vec{u},h))$ at times when we are only interested in the upper part of the forcing.

If we force below a condition $((\vec{u},h))$ such that $\dom h$ contains only sequences of length less than $\lambda$ then $\mathbb{R}_{\vec{u}}$ will add a generic sequence of the form $\langle \vec{w}_{\alpha}, g_{\alpha} \mid \alpha<\omega^{\lambda}\rangle$, where $g_i$ is generic in $\mathbb{B}(\vec{w}_{\alpha},\vec{w}_{\alpha+1})$. The $\omega^{\lambda}$ is ordinal exponentiation so as $\lambda$ is regular uncountable we in fact have $\omega^\lambda=\lambda$. This collapses all cardinals in the intervals $(\kappa(\vec{w}_{\alpha})^{+5},\kappa(\vec{w}_{\alpha+1}))$ and we shall see later that it preserves all other cardinals, so it will make $\kappa$ into $\aleph_{\lambda}$.

More generally, forcing with $\mathbb{R}_{\vec{u}}$ will add a generic sequence $\langle \vec{w}_{\alpha}, g_{\alpha} \mid \alpha<\theta+\lambda\rangle$
for some ordinal $\theta$.

\subsection{Properties of ultrafilter sequences II}

We are finally in a position to make the definition that we will use during the main construction.

\begin{definition}
	The class $\mathcal{U}$ is defined recursively to consist of all $\vec{u}\in\mathcal{U}'$ such that the $u_i$ concentrate on members of $\mathcal{U}$, and $\vec{u}$ satisfies the following additional property. (Note that the $h'\upharpoonright\vec{w}$ is given by restricting the domain of $h'$ to sequences $\vec{v}$ such that $\kappa(\vec{v})<\kappa(\vec{w})$ and $\lh\vec{v}<\lh\vec{w}$.)
	\begin{enumerate}[(1)]
		\setcounter{enumi}{5}
		\item (Capturing) Let $h$ be an upper part for $\mathbb{R}_{\vec{u}}$ and $X$ a $^*$-open set of lower parts for $\mathbb{R}_{\vec{u}}$. Then there is an upper part $h' \leq h$ such that for all lower parts $s$ and all $i<\lh\vec{u}$ we have one of:
		\begin{enumerate}[(i)]
			\item For all $\vec{w}\in \dom h'_i$ there do not exist $e$ and $q\leq h'(\vec{w})$ such that $s\frown((\vec{w},e,q)) \in X$.
			\item For all $\vec{w}\in \dom h'_i$ and $\vec{x} \in \dom h'$ such that $\kappa(\vec{w})<\kappa(\vec{x})$ there are densely many $q$ in $\mathbb{B}(\vec{w},\vec{x})$ below $h'_i(\vec{w})$ such that $s\frown((\vec{w},h'\upharpoonright\vec{w},q)) \in X$.
		\end{enumerate}
	\end{enumerate}
\end{definition}

\begin{proposition}
	Let $\vec{u}$ be constructed from a supercompactness embedding $j$ as before. Then $\vec{u}\in\mathcal{U}$.
\end{proposition}

\begin{proof}
	Say $\vec{u}$ is of the form $\langle \kappa, u_i, \mathcal{F}_i \mid i < \lambda \rangle$ and the supercompact ultrafilter sequence used in the construction is $\langle z,u^*_i,H_i\mid i<\lambda\rangle$. We have already established that $\vec{u} \in \mathcal{U}'$ so it remains to prove capturing. For each lower part $s$ begin by defining witnesses $f^s$ such that for all $\vec{w}\in\dom h$, if there are $e$ and $q\leq h(\vec{w})$ such that $s\frown ((\vec{w},e,q))\in X$ then there is $q\leq h(\vec{w})$ such that $s\frown ((\vec{w},f^s(\vec{w}),q))\in X$.

	We may assume that each $h_i$ is of the form $b(H_i, B_i)$ for some $B_i\in u^*_i$. For each lower part $s$ and each $i<\lambda$ choose $H^s_i \leq H_i$ such that for all $\vec{w}^* \in \dom H^s_i$ for which there exists $q \leq H_i(\vec{w}^*)$ with $s\frown ((\pi(\vec{w}^*),f^s(\pi(\vec{w}^*)),q))\in X$ we have $s \frown ((\pi(\vec{w}^*),f^s(\pi(\vec{w}^*)),H^s_i(\vec{w}^*)))\in X$. By normality take $H'_i$ such that for all $i$ and $s$ we have $H'_i \leq H^s_i\downharpoonright s$. For each $i<\lambda$ and lower part $s$ we can choose $C^s_i \subseteq B_i$ a member of $u^*_i$ such that one of the following occurs:
	\begin{enumerate}[(i)]
		\item For every $\vec{w}^*$ in $C^s_i$ there does not exist a $q \leq H_i(\vec{w}^*)$ such that $s\frown ((\pi(\vec{w}^*),f^s(\pi(\vec{w}^*)),q)) \in X$.
		\item For every $\vec{w}^*$ in $C^s_i$ we have $s\frown ((\pi(\vec{w}^*),f^s(\pi(\vec{w}^*)),H'_i(\vec{w}^*)))\in X$.
	\end{enumerate}
	Define $C_i := \triangle_s C^s_i$ and $h'_i:=b(H'_i, C_i)$. Observe that by construction $\Fil(H_i)$ is an ultrafilter and so equal to $\Fil(H'_i)$, whence $h'_i\in\mathcal{F}_i$. We can now prove a weaker version of the desired dichotomy.

	\begin{claim}
		Let $s$ be a lower part and $i<\lambda$. Then we have one of:
			\begin{enumerate}[(i)]
				\item For all $\vec{w}\in \dom h'_i$ there do not exist $e$ and $q\leq h'(\vec{w})$ such that $s\frown((\vec{w},e,q)) \in X$.
				\item For all $\vec{w}\in \dom h'_i$ there are densely many $q$ in $\mathbb{B}(\vec{w},\vec{u})$ below $h'_i(\vec{w})$ such that $s\frown((\vec{w},f^s(\vec{w}),q)) \in X$.
			\end{enumerate}
	\end{claim}

	\begin{proof}
		Suppose (i) is false, so we have $\vec{w}\in\dom h'_i$, $e$ and $q' \leq h'(\vec{w})\leq h(\vec{w})$ such that $s\frown((\vec{w},e,q'))\in X$. The choice of $f^s$ then gives us $q\leq h(\vec{w})$ such that $s\frown((\vec{w},f^s(\vec{w}),q)) \in X$. Now
		$$q \leq h_i(\vec{w}) = b(H_i, B_i)(\vec{w}) = \bigvee_{\pi(\vec{w}^*)=\vec{w}, \vec{w}^*\in B_i} H_i(\vec{w}^*)$$
		so there must be some $\vec{w}^* \in B_i$ with $\pi(\vec{w}^*)=\vec{w}$ such that $q \parallel H_i(\vec{w}^*)$. But $\vec{w}\in\dom h'_i$ so $\vec{w}^*\in C_i\subseteq C^s_i\downharpoonright s$; and $X$ is downwards closed so we cannot have been in the first case when we defined $C^s_i$, and must therefore be in the second case.

		We wish to show that (ii) holds, so we are given some $\vec{w}\in \dom h'_i$ and $r \in \mathbb{B}(\vec{w},\vec{u})$ below $h'_i(\vec{w})$. By similar reasoning we have that $r$ is compatible with $H'_i(\vec{w}^*)$ for some $\vec{w}^* \in C_i$ with $\pi(\vec{w}^*)=\vec{w}$. By the definition of $C^s_i$ we know that $s\frown ((\vec{w},f^s(\vec{w}),H'_i(\vec{w}^*)))\in X$ so it is possible to take $q\leq r$ with $s\frown ((\vec{w},f^s(\vec{w}),q))\in X$.
	\end{proof}

	For each lower part $s$ and each $i<\lambda$ that falls into case (ii) of the claim, and for each $\vec{w}\in \dom h'_i$ we have a dense open set of $q \in \mathbb{B}(\vec{w},\vec{u})$ such that $s\frown ((\vec{w},f^s(\vec{w}),q))\in X$, and we take a maximal antichain contained in both this set and $\mathbb{C}(\vec{w},\vec{u})$. The $\kappa$-chain condition of the forcing tells us that this antichain is bounded, which is to say there is some $\eta_{s, i, \vec{w}} < \kappa$ with the antichain contained in $\mathbb{C}(\vec{w},\eta_{s, i, \vec{w}})$. We now refine $\dom h'$ to contain only $\vec{x}$ such that $\kappa(\vec{x})$ is a closure point of the function $(s, i, \vec{w}) \mapsto \eta_{s, i, \vec{w}}$ and immediately have the following strengthening of the claim.

	For all lower parts $s$ and $i < \lambda$ we have one of:
	\begin{enumerate}[(i)]
		\item For all $\vec{w}\in \dom h'_i$ there do not exist $e$ and $q\leq h'(\vec{w})$ such that $s\frown((\vec{w},e,q)) \in X$.
		\item For all $\vec{w}\in \dom h'_i$ and $\vec{x}\in\dom h'$ such that $\kappa(\vec{w})<\kappa(\vec{x})$ there are densely many $q$ in $\mathbb{B}(\vec{w},\vec{x})$ below $h'_i(\vec{w})$ such that $s\frown((\vec{w},f^s(\vec{w}),q)) \in X$.
	\end{enumerate}

	To conclude the proof we will need to make further reductions of the domains of the $h'_i$. For each lower part $s$ and each $i<k<\lambda$ define $d^s_{i,k}$ to be the function $j(f^s)(\vec{u}\upharpoonright k)$ restricted to lower parts of length $i$. We observe that $j(f^s)(\vec{u}\upharpoonright k)$ is an upper part for $\mathbb{R}_{\vec{u}\upharpoonright k}$ so $d^s_{i,k}$ will have domain in $u_i$ and is a partial function from $V_{\kappa}$ to $V_{\kappa}$. Thus for any $\vec{v}$ in its domain we have $j(d^s_{i,k})(\vec{v}) = j(d^s_{i,k})(j(\vec{v})) = j(d^s_{i,k}(\vec{v}))=d^s_{i,k}(\vec{v})$, giving us
	\begin{align*}
		& \forall\vec{v}\in\dom d^s_{i,k}: j(d^s_{i,k})(\vec{v}) = d^s_{i,k}(\vec{v}) = j(f^s)(\vec{u}\upharpoonright k)(\vec{v}) \\
		\Rightarrow& \forall_{u_k}\vec{w}: \forall\vec{v}\in\dom d^s_{i,k}\cap V_{\kappa(\vec{w})}: d^s_{i,k}(\vec{v}) = f^s(\vec{w})(\vec{v})
	\end{align*}
	Call this $u_k$-large set $X^s_{i,k}$ and take $h''^s\leq h'$ such that for all $j<\lambda$ we have $\dom h''^s_j\subseteq (\bigcap_{k>j}\dom d^s_{j,k})\cap(\bigcap_{i<j}X^s_{i, j})$. Also ensure $h''^s_i\leq j(f^s)(\vec{u}\upharpoonright k)$ for all $i<k<\lambda$; this is possible since all the functions involved are members of the $\kappa$-complete filter $\mathcal{F}_i$. Then by normality take $h''$ such that $h''\downharpoonright s\leq h''^s$ for all $s$. For any lower part $s$, $\vec{w}\in\dom h''$ and $\vec{v}\in\dom h''\upharpoonright\vec{w}$ above $s$ this gives
	$$h''(\vec{v})\leq h''^s(\vec{v})\leq	j(f^s)(\vec{u}\upharpoonright (\lh\vec{w}))(\vec{v})=f^s(\vec{w})(\vec{v}).$$
	Hence $h''\upharpoonright\vec{w}\leq f^s(\vec{w})$ and since $X$ is $^*$-open we get $s\frown((\vec{w},h''\upharpoonright\vec{w},q))\in X$ for densely-many $q$ as required.
\end{proof}

This lemma is valuable because it allows us to express the crucial properties of $\vec{u}$ solely in terms of subsets of $V_{\kappa}$, rather than large supercompactness embeddings. When we perform the forcing iteration it will be possible to reflect these properties from the $\vec{u}$ that occurs at the end of the iteration to the $\vec{u}$ at earlier stages.

\section{Properties of the Radin forcing $\mathbb{R}_{\vec{u}}$} \label{properties_of_radin}

\subsection{The Prikry property}

\begin{proposition}
	Let $\vec{u}\in\mathcal{U}$. Then $\mathbb{R}_{\vec{u}}$ has the Prikry property.
\end{proposition}

\begin{proof}
	We are given some condition
	$$p = ((\vec{w}_0,h_0,p_0),...,(\vec{w}_{n-1},h_{n-1},p_{n-1}),(\vec{u},h_n))$$
	from $\mathbb{R}_{\vec{u}}$ and a proposition $\varphi$, and wish to find a direct extension of $p$ that forces either $\varphi$ or $\neg\varphi$. For notational convenience we will deem $\vec{w}_n$ to be $\vec{u}$.

	We define a descending sequence of $p^k$ by induction on $k\leq n$; starting with $p^0\leq^* p$ such that $p^0\parallel\phi$ if possible, or else $p^0:=p$. Given $p^{k-1}$, for each lower part $s\leq p^{k-1}\upharpoonright k$ for $\mathbb{R}_{\vec{w}_k}$, if possible take $((\vec{w}_k,h^s_k,p^s_k))\frown y^s_k \leq^* p^{k-1}\downharpoonright k$ such that $s\frown((\vec{w}_k,h^s_k,p^s_k))\frown y^s_k\parallel\phi$. Then by normality we can form $h'_k\leq h^s_k\downharpoonright s$ for all lower parts $s$, and by closure we can form $p'_k\leq p^s_k$ and $y'_k\leq y^s_k$ for all such $s$. Define $p^k := p^{k-1}\upharpoonright k\frown((\vec{w}_k,h'_k,p'_k))\frown y'_k$. The construction concludes with $p':=p^n$, so
	$$p' = ((\vec{w}_0,h'_0,p'_0),...,(\vec{w}_{n-1},h'_{n-1},p'_{n-1}),(\vec{w}_n,h'_n))$$
	such that for all $k\leq n$ and all lower parts $s \in \mathbb{R}_{\vec{w}_k}$, if there is some direct extension $t$ of $p'\downharpoonright k$ with $s\frown t\parallel\varphi$ then already $s\frown p'\downharpoonright k\parallel\varphi$. Then for each $k\leq n$ define $X^+_k$ to be the set of lower parts $s$ in $\mathbb{R}_{\vec{w}_k}$ such that $s\frown p'\downharpoonright k\forces\varphi$. Similarly define $X^-_k$ with $\neg\varphi$ in place of $\varphi$. Take $h''_k \leq h'_k$ that captures both $X^+_k$ and $X^-_k$.

	We claim that $$p'':=((\vec{w}_0,h''_0,p'_0),...,(\vec{w}_{n-1},h''_{n-1},p'_{n-1}),(\vec{w}_n,h''_n))$$
	decides $\varphi$. Suppose this is not so, and take $t$ of minimal length below $p''$ that decides $\varphi$; without loss of generality we can assume $t\forces\varphi$. Fix $k\leq n$ such that the largest new triple appearing in $t$ lies between $\vec{w}_{k-1}$ and $\vec{w}_k$. Call this triple $(\vec{v},e,q)$ and split $t$ as $r\frown((\vec{v},e,q))\frown s$. Observe that by the construction we actually have $r\frown((\vec{v},e,q))\frown p''\downharpoonright k \forces \varphi$. Observe further that the existence of such a $(\vec{v},e,q)$ tells us that $r$ and $\epsilon:=\lh\vec{v}$ fall into case (ii) of the capture of $X^+_k$. We will show by density that in fact $r\frown p''\downharpoonright k\forces\varphi$, which will contradict the minimality of the length of $t$ and conclude the proof.

	We are given some extension of $r\frown p''\downharpoonright k$, say of the form
	$$r'\frown ((\vec{v}_0,e_0,q_0),...,(\vec{v}_{m-1},e_{m-1},q_{m-1}))\frown y$$
	where $\kappa(\max r') = \kappa(\max r)$ and $\kappa(\min y) = \kappa(\vec{w}_k)$, and we seek an extension that forces $\varphi$. Fix $j$ such that $\lh\vec{v}_j = \epsilon$ and $\lh\vec{v}_i < \epsilon$ for all $i<j$; if there is no such $\vec{v}_j$ then we can easily insert one. We have $\vec{v}_j \in \dom h''_k$ and $q_j\leq h''_k(\vec{v}_j)$, and case (ii) of the capturing of $X^+_k$ occurs for $r$ and $\epsilon$, so we can find $q^*\leq q_j$ such that $r\frown((\vec{v}_j,h''_k\upharpoonright\vec{v}_j,q^*))\in X^+_k$, which is to say
	$$r\frown((\vec{v}_j,h''_k\upharpoonright\vec{v}_j,q^*))\frown p'\downharpoonright k\forces\varphi.$$
	For $i<j$ the fact that $(\vec{v}_i,e_i,q_i)$ could be added below $(\vec{w}_k,h''_k)$ shows us that it can also be added below $(\vec{v}_j,h''_k\upharpoonright\vec{v}_j)$. This establishes that
	$$r'\frown ((\vec{v}_0,e_0,q_0),... (\vec{v}_j,e_j\wedge h''_k\upharpoonright\vec{v}_j, q^*) ,...,(\vec{v}_{m-1},e_{m-1},q_{m-1}))\frown y$$
	is below $r\frown((\vec{v}_j,h''_k\upharpoonright\vec{v}_j,q^*))\frown p'\downharpoonright k$ and hence forces $\varphi$, and it is also an extension of $r'\frown ((\vec{v}_0,e_0,q_0),...,(\vec{v}_{m-1},e_{m-1},q_{m-1}))\frown y$ as required.
\end{proof}

We can use this result to show that $\mathbb{R}_{\vec{u}}$ preserves enough cardinals.

\begin{proposition} \label{preserveCardinals}
	\begin{enumerate}[(a)]
		\item Let $\vec{u}\in\mathcal{U}$ and $\langle\vec{w}_{\alpha},g_{\alpha}\mid\alpha<\theta\rangle$ the generic sequence of ultrafilter sequences and collapses added by $\mathbb{R}_{\vec{u}}$.

		Then for $\alpha<\theta$, $\mathbb{R}_{\vec{u}}$ preserves the cardinals in $[\kappa(\vec{w}_{\alpha}),\kappa(\vec{w}_{\alpha})^{+5}]$.
		\item If we force below $((\vec{u},h))$ such that $\dom h$ contains only sequences of length less than $\lh\vec{u}$ then $\kappa$ becomes $\aleph_{\lh\vec{u}}$.
	\end{enumerate}
\end{proposition}

\begin{proof}
	\begin{enumerate}[(a)]
		\item Our proof is by induction on $\kappa(\vec{u})$. Given $\alpha<\theta$ take a condition $p$ in the generic filter of the form $p_1\frown((\vec{w}_{\alpha},e,q))\frown p_2$. Below $p$, $\mathbb{R}_{\vec{u}}$ splits as
		$$\mathbb{R}_{\vec{w}_{\alpha}}/p_1\frown((\vec{w}_{\alpha},e)) \:\:\times\:\: \mathbb{R}'_{\vec{u}}/((\langle\kappa(\vec{w}_{\alpha})\rangle,\phi,q))\frown p_2$$
		where $\mathbb{R}'_{\vec{u}}$ is the same as $\mathbb{R}_{\vec{u}}$ except with its first collapse starting from $\kappa(\vec{w}_{\alpha})^{+5}$ instead of $\omega^{+5}$. By hypothesis the first of these forcings preserves many cardinals below $\kappa(\vec{w}_{\alpha})$ and hence $\kappa(\vec{w}_{\alpha})$ itself; since it has the $\kappa(\vec{w}_{\alpha})^+$-cc it also preserves all larger cardinals. The second forcing has the Prikry property by a proof identical to the one above, and it is $\kappa(\vec{w}_{\alpha})^{+5}$-closed in the $\leq^*$-ordering, so it will preserve all remaining cardinals up to and including $\kappa(\vec{w}_{\alpha})^{+5}$.
		\item Here we have $\theta=\omega^{\lh\vec{u}}=\lh\vec{u}$. By part (a) we have that cardinals in $[\kappa(\vec{w}_{\alpha}),\kappa(\vec{w}_{\alpha})^{+5}]$ are preserved, and it is clear that all other cardinals below $\kappa(\vec{u})$ are collapsed. So the forcing leaves $\lh\vec{u}$-many cardinals below $\kappa$.
	\end{enumerate}
\end{proof}

\subsection{Analysis of names}

Next we prove a technical lemma allowing us to replace $\mathbb{R}_{\vec{u}}$-names with names in smaller forcings; it will be useful to us in Lemma \ref{Qstrong} when we need to establish tight control over such names.

\begin{lemma} \label{nameAnalysis}
	Let $\vec{u}\in\mathcal{U}$, $\dot{x}$ a Boolean $\mathbb{R}_{\vec{u}}$-name (i.e. a name for a single true/false value), and $s\frown((\vec{u},h))\in\mathbb{R}_{\vec{u}}$.

	Then there is an ordinal $\beta<\kappa$, an upper part $h'\leq h$, and a $\mathbb{R}_{\max s}\times\mathbb{B}(\max s, \beta)$-name $\dot{y}$ such that $\dom h'$ lies above $\beta$ (so below $((\vec{u},h'))$ this $\dot{y}$ can be regarded as a $\mathbb{R}_{\vec{u}}$-name) and such that $s\frown((\vec{u},h'))\forces \dot{x}=\dot{y}$.
\end{lemma}

\begin{proof}
	For each $\beta>\kappa(\max s)$ and each $\mathbb{R}_{\max s}\times\mathbb{B}(\max s,\beta)$-name $\dot{y}$ use normality to take $h_{\dot{y}}\leq h$ such that for every lower part $t$, if there is some $h^*\leq h$ with $t\frown((\vec{u},h^*))\forces\dot{x}=\dot{y}$ then $t\frown((\vec{u},h_{\dot{y}}))\forces\dot{x}=\dot{y}$. Define $X_{\dot{y}}$ to be the set of lower parts $t$ such that
	$$t\frown((\vec{u},h_{\dot{y}}))\forces\dot{x}=\dot{y}$$
	and take $h'_{\dot{y}}\leq h_{\dot{y}}$ capturing $X_{\dot{y}}$. Then use normality again to get $h'$ such that $h'\downharpoonright (\beta+1)\leq h_{\dot{y}}$ for all $\mathbb{R}_{\max s}\times\mathbb{B}(\max s,\beta)$-names $\dot{y}$.

	Take $\vec{w}\in\dom h'$ above $s$ and of length $0$ (i.e. $\vec{w}=\langle\kappa(\vec{w})\rangle$). We can split $\mathbb{R}_{\vec{u}}$ below $s\frown((\vec{w},\phi,0),(\vec{u},h'))$ as
	$$\mathbb{R}_{\vec{w}}/s\frown((\vec{w},\phi))\times \mathbb{R}'_{\vec{u}}/((\vec{w},\phi, 0),(\vec{u},h'\downharpoonright \kappa(\vec{w})))$$
	where $\mathbb{R}'_{\vec{u}}$ is the usual forcing derived from $u$ except that its first collapse starts from $\kappa(\vec{w})^{+5}$ rather than $\omega^{+5}$. Now we can view $\dot{x}$ as being a $\mathbb{R}'_{\vec{u}}$-name for a $\mathbb{R}_{\vec{w}}$-name. We know that $\mathbb{R}_{\vec{w}}$ has the $\kappa(\vec{w})^+$-cc, so the $\mathbb{R}_{\vec{w}}$-name in question consists of at most $\kappa(\vec{w})$-many pieces of information. This allows us to use the Prikry property and closure of $\mathbb{R}'_{\vec{u}}$ to take a direct extension $((\langle\omega\rangle,\phi,q), (\vec{u},h''))$ of $((\langle\omega\rangle,\phi,0), (\vec{u},h'\downharpoonright \kappa(\vec{w})))$ that determines the value of the $\mathbb{R}_{\vec{w}}$-name, say as $\dot{y}$, a $\mathbb{R}_{\vec{w}}$-name in the ground model. So returning to $\mathbb{R}_{\vec{u}}$ we have
	$$s\frown((\vec{w},\phi,q),(\vec{u},h''))\forces\dot{x}=\dot{y}.$$

	Observe that (since $\lh\vec{w}=0$) $\mathbb{R}_{\vec{w}}$ splits below $s\frown((\vec{w},\phi))$ as $\mathbb{R}_{\max s}/s \times \mathbb{B}(\max s, \vec{w})$; this has the $\kappa(\vec{w})$-cc so $\dot{y}$ is in fact a $\mathbb{R}_{\max s}\times\mathbb{B}(\max s, \beta)$-name for some $\beta<\kappa(\vec{w})$. Now by construction $X_{\dot{y}}$ contains $s\frown((\vec{w},\phi,q))$, and $\vec{w}\in\dom h_{\dot{y}}$, so when we captured $X_{\dot{y}}$ we must have been in case (ii) for $s$ and $0$. We can now use the same argument as in the proof of the Prikry condition to show that any extension of $s$ must be extensible to some condition in $X_{\dot{y}}$, so we have that $s\frown((\vec{u},h'))\forces\dot{x}=\dot{y}$ as required.
\end{proof}

\subsection{Characterisation of genericity}

Finally we look for a way to characterise genericity that will allow us to take generic sequences for one Radin forcing and show that they are also generic for other Radin forcings. This characterisation develops similar ideas for simpler forcings found in \cite{Mathias} and \cite{Mitchell}.

\definition
Let $\vec{u}\in\mathcal{U}$. A sequence $\langle \vec{w}_{\alpha}, g_{\alpha} \mid \alpha < \theta\rangle$ in some outer model of set theory is {\em geometric} for $\mathbb{R}_{\vec{u}}$ if it satisfies:
\begin{enumerate}
	\item $\{ \kappa(\vec{w}_{\alpha}) \mid \alpha < \theta\}$ is club in $\kappa(\vec{u})$ with $\kappa(\vec{w}_0)=\omega$.
	\item For all limit ordinals $\alpha < \theta$, $\langle \vec{w}_{\beta}, g_{\beta} \mid \beta < \alpha\rangle$ is generic for $\mathbb{R}_{\vec{w}_{\alpha}}$.
	\item For all $\alpha$, $g_{\alpha}$ is $\mathbb{B}(\vec{w}_{\alpha},\vec{w}_{\alpha+1})$-generic.
	\item For every $X\in V_{\kappa(\vec{u})+1}$; $X \in \bigcap_{i<\lh\vec{u}}u_i$ iff for all large $\alpha$, $\vec{w}_{\alpha}\in X$.
	\item For every upper part $h$ for $\mathbb{R}_{\vec{u}}$, for all large $\alpha$, $h(\vec{w}_{\alpha})\in g_{\alpha}$.
\end{enumerate}

Note (4) implies that for all $i<\lh\vec{u}$ there are unboundedly many $\alpha<\theta$ such that $\lh(\vec{w}_{\alpha})=i$.

It is clear that a generic sequence for $\mathbb{R}_{\vec{u}}$ is geometric; we aim to show the converse.

\begin{definition} \label{defnGenericFilter}
	We have already seen that from a generic filter for $\mathbb{R}_{\vec{u}}$ it is possible to derive a generic sequence $G=\langle \vec{w}_{\alpha}, g_{\alpha} \mid \alpha < \theta\rangle$. Conversely, given such a generic sequence we can rederive the generic filter $F_G$, which will consist of all conditions $((\vec{v}_0,e_0,p_0),...,(\vec{v}_{n-1},e_{n-1},p_{n-1}),(\vec{u},h))$ with $(\vec{v}_n,e_n):=(\vec{u},h)$ such that:
	\begin{itemize}
		\item For all $k<n$ there is $\alpha$ such that $\vec{v}_k=\vec{w}_{\alpha}$ and  $p_k\in g_{\alpha}$
		\item For $\alpha$ and $k\leq n$, if $\kappa(\vec{v}_{k-1})<\kappa(\vec{w}_{\alpha})<\kappa(\vec{v}_k)$, then $\vec{w}_{\alpha}\in\dom e_k$ and $e_k(\vec{w}_{\alpha})\in g_{\alpha}$.
	\end{itemize}
\end{definition}

\begin{definition}
	We will say a sequence $\langle \vec{w}_{\alpha}, g_{\alpha} \mid \alpha < \theta\rangle$ {\em respects} $(\vec{u},h)$ if for all $\alpha < \theta$ we have $\vec{w}_{\alpha} \in \dom h$ and $h(\vec{w}_{\alpha})\in g_{\alpha}$.
\end{definition}

\begin{lemma} \label{geometricLemma}
	Let $\vec{u} \in \mathcal{U}$, $h$ an upper part for $\mathbb{R}_{\vec{u}}$ and $D\subseteq	\mathbb{R}_{\vec{u}}$ dense open. Then there is an upper part $h' \leq h$ such that for every geometric sequence $G$ respecting $(\vec{u},h')$ we have $F_G\cap D \neq \phi$.
\end{lemma}

\begin{proof}
	Define $\kappa:=\kappa(\vec{u})$ and $\lambda:=\lh\vec{u}$. Invoke normality to take $h^*\leq h$ such that for every lower part $t$, if there is an $h'$ such that $t\frown((\vec{u},h'))\in D$ then $t\frown((\vec{u},h^*\downharpoonright t))\in D$.	Define $X^{\phi}$ to be the set of lower parts $t$ such that $t\frown((\vec{u},h^*))\in D$. We will inductively define $X^{\eta}$ for $\eta$ any finite sequence of $i<\lambda$. First take $h^{\eta}\leq h^*$ capturing $X^{\eta}$. Then for each $i<\lambda$, the set $X^{\langle i\rangle\frown\eta}$ will consist of all $t$ that with $i$ fall into case (ii) of the capturing of $X^{\eta}$. The number of possible $\eta$ is less than $\kappa$ so by $\kappa$-completeness we can fix $h'$ that is below $h^{\eta}$ for all such $\eta$.

	We note that there are densely-many $r$ in $\mathbb{B}(\omega,\kappa)$ such that $((\langle\omega\rangle,\phi,r))\in X^{\eta}$ for some $\eta$. This is because for any $r\in\mathbb{B}(\omega,\kappa)$ we can extend $((\langle\omega\rangle,\phi,r),(\vec{u},h'))$ to a condition $t\frown((\vec{u},h^*)) \in D$; then $t\in X^{\phi}$ and inductively removing triples of $t$ from the right yields $((\langle\omega\rangle,\phi,r'))\in X^{\eta}$ for some $r'\leq r$ and $\eta$.

	Now we are given a geometric sequence $G=\langle \vec{w}_{\alpha}, g_{\alpha} \mid \alpha < \theta\rangle$ that respects $(\vec{u},h')$ and must show $F_G\cap D \neq \phi$.

	By the density in $\mathbb{B}(\omega,\kappa)$ just noted, and by the genericity of $g_0$, take $q_0\in g_0$ and $\eta$ such that $((\langle\omega\rangle,\phi,q_0))\in X^{\eta}$. Say $\eta=:\langle i_1,...,i_{n-1}\rangle$. Define $\alpha_0=0$ and then inductively take $\alpha_{k+1}>\alpha_k$ minimal such that $\lh\vec{w}_{\alpha_{k+1}}=i_k$. We note that this must be possible by clause (4) in the definition of geometricity. Then for $k\geq 1$ inductively choose $q_k\in g_{\alpha_k}$ such that
	$$s_k:=((\vec{w}_0,h'\upharpoonright\vec{w}_0,q_0),...,(\vec{w}_{\alpha_k},h'\upharpoonright\vec{w}_{\alpha_k},q_k))\in X^{\langle i_{k+1},...i_{n-1}\rangle},$$
	invoking the nature of case (ii) capturing, the genericity of the $g_k$, and the fact $h'(\vec{w}_{\alpha_k})\in g_{\alpha_k}$. This concludes with $s_{n-1}\in X^{\phi}$ so $s_{n-1}\frown((\vec{u},h'))\in D$ (as $h'\leq h^*$), and it remains to show that $s_{n-1}\frown((\vec{u},h'))\in F_G$.

	The first of the two requirements from Definition \ref{defnGenericFilter} is clear from the construction itself. Now we must consider the case of some $\beta < \theta$ and $k\leq n$ such that $\kappa(\vec{w}_{\alpha_{k-1}}) < \kappa(\vec{w}_{\beta}) < \kappa(\vec{w}_{\alpha_k})$, or equivalently $\alpha_{k-1}<\beta<\alpha_k$. Note that the $k=n$ case is taken care of by the respect of $G$ for $(\vec{u},h')$. The minimality of our choice of $\alpha_k$ (together with clauses 2 and 4 in the definition of geometric) tells us that $\lh\vec{w}_{\beta}<\lh\vec{w}_{\alpha_k}$, so $\vec{w}_{\beta}\in\dom h'\upharpoonright\vec{w}_{\alpha_k}$. Then from the respect of $G$ for $(\vec{u},h')$ we have $h'(\vec{w}_{\beta})\in g_{\beta}$ as required.
\end{proof}

\begin{proposition} \label{characteriseGenericity}
	Let $\vec{u}\in\mathcal{U}$. Then a sequence $G$ is generic for $\mathbb{R}_{\vec{u}}$ iff it is geometric for $\mathbb{R}_{\vec{u}}$.
\end{proposition}

\begin{proof}
	We are given $D\subseteq\mathbb{R}_{\vec{u}}$ dense open and wish to show that $D\cap F_G\neq\phi$. We begin by using normality to take an upper part $h$ such that for all lower parts $s$, if there is some $\tilde{h}$ such that $s\frown((\vec{u},\tilde{h}))\in D$ then already $s\frown((\vec{u},h\downharpoonright s))\in D$. For each lower part $s$, use the Prikry property for $\mathbb{R}_{\vec{u}}$ to take $h_s\leq h\downharpoonright s$ such that
	$$((\vec{u},h_s))\parallel\exists t\in \dot{\Gamma}: s\frown t\frown ((\vec{u},h))\in D$$
	where $\dot{\Gamma}$ is the name for the set of all lower parts for $\mathbb{R}_{\vec{u}}$ that appear as the lower part of some condition in the generic filter. We say $s$ is {\em good} if the decision is positive. For good $s$ we have $$D_s:=\{t\frown((\vec{u},h^*))\mid s\frown t\frown((\vec{u},h^*))\in D\}$$
	dense open below $((\vec{u},h_s))$. For these $s$ use Lemma \ref{geometricLemma} to take $h'_s\leq h_s$ such that for all geometric $G$ respecting $(\vec{u},h'_s)$ we have $F_G\cap D \neq\phi$. Then by normality take $h'\leq h\downharpoonright s$ for all $s$, and note that also
	$$((\vec{u},h'))\parallel	\exists t\in \dot{\Gamma}: s\frown t\frown ((\vec{u},h))\in D$$
	for all lower parts $s$. Finally take $h''\leq h'$ that captures the set of good lower parts; we say that a lower part $r$ that falls into case (ii) of this capturing is {\em pre-good}.

	Write the geometric $G$ we were given as $\langle \vec{w}_{\alpha}, g_{\alpha} \mid \alpha < \theta\rangle$, and use clause 5 of geometricity to take $\beta<\theta$ a limit ordinal such that
	$$\forall \alpha\geq \beta: \vec{w}_{\alpha}\in\dom h'', h''(\vec{w}_{\alpha})\in g_{\alpha}.$$
	Then take $\gamma>\beta$ also limit such that for all $\beta<\alpha<\gamma$, $\lh\vec{w}_{\alpha}<\lh\vec{w}_{\gamma}$; this is possible as $\cf\theta>\omega$ and all lengths occur cofinally in $\langle\vec{w}_{\alpha}\mid\alpha<\theta\rangle$.

	\begin{claim}
		There are densely-many $r\frown((\vec{w}_{\gamma},e))$ in $\mathbb{R}_{\vec{w}_{\gamma}}$ such that $r$ is pre-good.
	\end{claim}

	\begin{proof}
		We are given a condition $r\frown((\vec{w}_{\gamma},e))\in\mathbb{R}_{\vec{w}_{\gamma}}$ and have that $r\frown((\vec{w}_{\gamma},e,0),(\vec{u},h'')) \in \mathbb{R}_{\vec{u}}$ so we can extend it to
		$$r'\frown((\vec{w}_{\gamma},e',q))\frown t\frown((\vec{u},h''))\in D,$$
		and the choice of $h$ then gives
		$$r'\frown((\vec{w}_{\gamma},e',q))\frown t\frown((\vec{u},h))\in D.$$
		The decision made by $((\vec{u},h'))$ for $r'\frown((\vec{w}_{\gamma},e',q))$ must thus have been positive, which is to say $r'\frown((\vec{w}_{\gamma},e',q))$ is good. The fact that $\vec{w}_{\gamma}\in\dom h''$ and $q\leq h''(\vec{w})$ then gives that $r'$ is pre-good; we are now done because $r'\frown((\vec{w}_{\gamma},e'))\leq r\frown((\vec{w}_{\gamma},e))$ in $\mathbb{R}_{\vec{w}_{\gamma}}$.
	\end{proof}

	The claim allows us to use property 2 of $G$ to find some pre-good $r$ and $e$ with $r\frown((\vec{w}_{\gamma},e))\in F_{G\upharpoonright\gamma}$ and $\kappa(\max(r))>\kappa(\beta)$. Then we use case (ii) of capturing to take $p\in\mathbb{B}(\vec{w}_{\gamma},\vec{w}_{\gamma+1})$ such that $p\in g_{\gamma}$ and $s:=r\frown((\vec{w}_{\gamma},h''\upharpoonright\vec{w}_{\gamma},p))$ is good. For all $\alpha<\gamma$ with $\kappa(\vec{w}_{\alpha})$ above $r$ we have $\alpha>\beta$, so $\lh\vec{w}_{\alpha}<\lh\vec{w}_{\gamma}$ and $\vec{w}_{\alpha}\in\dom h''\upharpoonright\vec{w}_{\gamma}$; also $h''(\vec{w}_{\alpha})\in g_{\alpha}$ by the choice of $\beta$. Combining these shows us that $s\frown((\vec{w}_{\gamma+1},\phi))\in F_{G\upharpoonright(\gamma+1)}$.

	Now $D_s$ is dense and $G\downharpoonright(\gamma+1)$ respects $h'$ so we can take $t\frown((\vec{u},h^*))\leq((\vec{u},h'))$ in $F_{G\downharpoonright(\gamma+1)}\cap D_s$. Then by the definitions of $F_G$ and $D_s$ we have $s\frown t\frown((\vec{u},h^*))\in F_G\cap D$, and are done.
\end{proof}

\section{The preparatory forcing $\mathbb{Q}_{\vec{u}}$} \label{preparatory_forcing}

In this section we work in the following context.

\begin{setting}
	Let $\vec{u}\in\mathcal{U}$ with $\kappa:=\kappa(\vec{u})$ and $\lambda:=\lh\vec{u}$ regular uncountable. Let $\kappa^{<\kappa}=\kappa$ and $2^{\kappa^+}=\kappa^{+3}$. Assume there exists a binary tree $T$ of height and size $\kappa^+$ (i.e. a tree such that each node has two successors on the next level) with $\langle x_{\alpha}\mid \alpha<\kappa^{+3}\rangle$ an enumeration of its branches. Let $\langle \dot{\mathcal{E}}_{\alpha}\mid \alpha<\kappa^{+3}\rangle$ be an enumeration of the $\mathbb{R}_{\vec{u}}$-names for graphs on $\kappa^+$. We note that such an enumeration is possible since $\mathbb{R}_{\vec{u}}$ has size $2^{\kappa}$ and the $\kappa^+$-cc.
\end{setting}

\subsection{Defining the forcing $\mathbb{Q}_{\vec{u}}$}

We will want to perform an iteration that preserves $V_{\kappa}$ and successively expands $V_{\kappa+1}$ and thus the sequences of ultrafilters. With this in mind we consider a member $\vec{u}$ of $\mathcal{U}$, and seek to add a partial function $g$ from $V_{\kappa}$ to $V_{\kappa}$ such that defining $g_i := g\upharpoonright\{\vec{w}\in\mathcal{U}\mid\lh \vec{w} = i\}$ we could potentially expand $\vec{u}$ to some $\vec{u}'$ in the generic extension with $g_i\in \mathcal{F}_{\vec{u}',i}$. In order to accomplish this we will need the $g$ we build to be appropriately compatible with the pre-existing members of $\vec{u}$, motivating the following definition which generalises long Prikry forcing to the case of Radin forcing with collapses.

\begin{definition}
	Let $\vec{u}\in\mathcal{U}$. Then $\mathbb{M}_{\vec{u}}$ is defined to have conditions $(c,h)$ where $h$ is an upper part for $\vec{u}$ and there is $\rho^{(c,h)}:=\rho<\kappa$ such that $c$ is a partial function from $\mathcal{U}\cap V_{\rho}$ to $V_{\kappa}$ such that
	$$\forall \vec{v}\in \dom c: c(\vec{v})\in\mathbb{B}(\vec{v},\vec{\kappa})-\{0\}, c\upharpoonright\kappa(\vec{v}) \in \mathcal{F}_{\vec{v}}.$$
	We also require $\kappa(\vec{v})<\kappa(\vec{w})$ for $\vec{v}\in\dom c$ and $\vec{w}\in\dom h$.

	We define $(c',h')\leq(c,h)$ if $c'\upharpoonright\rho^{(c,h)}=c$, $h'\leq h$ and for each $\vec{w}\in\dom c' - \dom c$ we have $\vec{w}\in\dom h$ and $c(\vec{w})\leq h(\vec{w})$.

	Also define $a^{(c,h)}:=\{\kappa(\vec{w})\mid \vec{w}\in\dom c\}$.
\end{definition}

For any $(c,h)\in\mathbb{M}_{\vec{u}}$ and $\vec{w}\in\dom h$ the definition of upper part gives us that $(c\cup h\upharpoonright(\kappa(\vec{w})+1), h\downharpoonright(\kappa(\vec{w})+1))\leq(c,h)$ so by density $\mathbb{M}_{\vec{u}}$ will add a partial function $g$ from $V_{\kappa}$ to $V_{\kappa}$ such that for all upper parts $h$ there is a $\mu<\kappa$ with $g\downharpoonright\mu\leq h$.

We now augment this definition into one that will help us add a family of universal graphs together with functions witnessing their universality.

\begin{definition}
	Let $\vec{u}\in\mathcal{U}$. Then $\mathbb{Q}^*_{\vec{u}}$ has conditions $p=(c,h,t,f)$ such that:
	\begin{enumerate}
		\item $(c,h)\in\mathbb{M}_{\vec{u}}$. We define $a^p$ to be $a^{(c,h)}$.
		\item $t\in[(a^p\cap\sup a^p)\times \kappa^{+3}]^{<\kappa}$.
		\item $f =: \langle f^{\eta}_{\alpha}\mid (\eta,\alpha)\in t\rangle$ with $\dom f^{\eta}_{\alpha}\in [\kappa^+]^{<\kappa}$.
		\item For each $(\eta,\alpha)\in t$ and $\zeta\in\dom f^{\eta}_{\alpha}$ there is $\gamma < \kappa$ with $f^{\eta}_{\alpha}(\zeta) = (x_{\alpha}\upharpoonright\zeta, \gamma)$.
	\end{enumerate}
	We also write $t^{\eta}:=\{\alpha \mid (\eta,\alpha)\in t\}$.

	We define $(c',h',t',f')\leq(c,h,t,f)$ if $(c',h')\leq(c,h)$ in $\mathbb{M}_{\vec{u}}$, $t'\supseteq t$, and for all $(\eta, \alpha)\in t$ we have $f'^{\eta}_{\alpha}\supseteq f^{\eta}_{\alpha}$.

	Note that this definition is implicitly dependent on the $\langle x_{\alpha}\mid\alpha<\kappa^{+3}\rangle$ and $\langle\dot{\mathcal{E}}_{\alpha}\mid\alpha<\kappa^{+3}\rangle$ from the setting.
\end{definition}

In addition to the function $g$ added by $\mathbb{M}_{\vec{u}}$, this forcing will for each $\vec{w}\in \dom g$ and $\alpha<\kappa^{+3}$ add a function from $\kappa^+$ to $T\times\kappa$, the first co-ordinate of which will run along the branch $x_{\alpha}$. The idea here is that after Radin forcing the $\mathbb{R}_{\vec{u}}$-name $\dot{\mathcal{E}}_{\alpha}$ will be realised as a graph on $\kappa^+$ and then (for some $\eta$ to be selected later) the function $f^{\eta}_{\alpha}$ will map it into $T\times\kappa$. We will then include in our list of jointly-universal graphs the graph on $T\times\kappa$ induced by all these embeddings. This raises the problem that there may be disagreements between the many graphs we are trying to simultaneously embed as to whether or not a particular edge should exist. In order to gain better control of the situation we will add a fifth requirement on forcing conditions, and for this we need a technical definition.

\begin{definition}
	Let $s=\langle (\vec{w}_k,e_k,q_k)\mid k<n\rangle$ be a lower part for $\mathbb{R}_{\vec{u}}$, $c$ the first co-ordinate of a condition from $\mathbb{M}_{\vec{u}}$ and $\eta<\kappa$. Then we say $s$ is {\em harmonious with $c$ past $\eta$} if for all $k<n$ we have one of:
	\begin{itemize}
		\item $\kappa(\vec{w}_k)<\eta$.
		\item $\kappa(\vec{w}_k)=\eta$ and $\lh\vec{w}_k = 0$.
		\item $\kappa(\vec{w}_k)>\eta$, $e_k\leq c\upharpoonright\kappa(\vec{w}_k)$, $\kappa(\vec{v})>\eta$ for all $\vec{v}\in\dom e_k$, $\vec{w}_k\in\dom c$ and $q_k\leq c(\vec{w}_k)$.
	\end{itemize}
\end{definition}

\begin{lemma} \label{openHarmony}
	Let $s$ be harmonious with $c$ past $\eta$ and $s'\leq s$ (so $\kappa(\max s')=\kappa(\max s)$). Then $s'$ is harmonious with $c$ past $\eta$.
\end{lemma}

\begin{proof}
	Consider some element $(\vec{v},d,p)$ from $s'$. If $\vec{v}$ already occurs in $s$ then it is clear that is satisfies the conditions. Otherwise it was added below some element $(\vec{w},e,q)$ from $s$ (because $\kappa(\max s') = \kappa(\max s)$). If $\kappa(\vec{w})\leq\eta$ then $\kappa(\vec{v})<\eta$ so all is well. Otherwise since $\vec{v}\in\dom e$ we get $\kappa(\vec{v})>\eta$. The required conditions in this case follow since $d\leq e\upharpoonright V_{\kappa(\vec{v})}$ and $p\leq e(\vec{v})$.
\end{proof}

We are now ready to define the desired forcing.

\begin{definition}
	The forcing $\mathbb{Q}_{\vec{u}}$ consists of conditions $(c,h,t,f)$ that satisfy the four conditions from the definition of $\mathbb{Q}^*_{\vec{u}}$, which for convenience we repeat here, together with one more:
	\begin{enumerate}
		\item $(c,h)\in\mathbb{M}_{\vec{u}}$.
		\item $t\in[(a\cap\sup a)\times \kappa^{+3}]^{<\kappa}$ where $a:=a^{(c,h)}$.
		\item $f =: \langle f^{\eta}_{\alpha}\mid (\eta,\alpha)\in t\rangle$ with $\dom f^{\eta}_{\alpha}\in [\kappa^+]^{<\kappa}$.
		\item For each relevant $\eta$, $\alpha$ and $\zeta$ there is $\gamma < \kappa$ with $f^{\eta}_{\alpha}(\zeta) = (x_{\alpha}\upharpoonright\zeta, \gamma)$.
		\item Let $\eta\in a\cap\sup a$, $\alpha,\beta\in t^{\eta}$, $s$ a lower part for $\mathbb{R}_{\vec{u}}$ that is harmonious with $c$ past $\eta$, and $\zeta,\zeta'\in\dom f^{\eta}_{\alpha}\cap \dom f^{\eta}_{\beta}$. Let also $f^{\eta}_{\alpha}(\zeta) = f^{\eta}_{\beta}(\zeta)\neq f^{\eta}_{\alpha}(\zeta')=f^{\eta}_{\beta}(\zeta')$. Then
			$$s\frown((\vec{u},h))\forces_{\mathbb{R}_{\vec{u}}} \zeta \dot{\mathcal{E}}_{\alpha}\zeta' \leftrightarrow \zeta\dot{\mathcal{E}}_{\beta}\zeta'.$$
	\end{enumerate}
\end{definition}

We will be able to use this final condition at the end of the argument to ensure that graphs (given by the $\dot{\mathcal{E}}_{\alpha}$) that we wish to map to the same place will agree about which edges should exist. But first we must establish that the right kind of generic object is still added. In doing so we establish a slightly stronger result that tidies up the conditions and will aid some of our later reasoning.

\begin{lemma} \label{squareOff}
	Let $\vec{u}\in\mathcal{U}$, $l$ an upper part for $\mathbb{R}_{\vec{u}}$, $\eta<\mu<\kappa$, $\epsilon,\epsilon'<\kappa^{+3}$ and $\zeta,\zeta'<\kappa^+$. Let $p \in \mathbb{Q}_{\vec{u}}$ with $\eta \in a^p$. Then there are densely many conditions $q=(c,h,t,f)$ below $p$ such that $a^q$ has a maximal element greater than $\mu$, $h\leq l$, and there are $A\in[\kappa^{+3}]^{<\kappa}$ and $B\in[\kappa^+]^{<\kappa}$ with $t = (a^q\cap\sup a^q)\times A\ni(\eta,\epsilon),(\eta,\epsilon')$ and $\dom f^{\theta}_{\beta} = B\ni\zeta,\zeta'$ for each $(\theta, \beta)\in t$.
\end{lemma}

\begin{proof}
	We are given some condition $r=(c,h,t,f)\leq p$ to extend. Choose some $\vec{w}\in\dom h$ such that $\kappa(\vec{w})>\mu$ and define $c' = c\cup h\upharpoonright (\kappa(\vec{w})+1)$ and take $h'\leq l$, $h\downharpoonright\vec{w}$. Then $(c',h')$ will be in $\mathbb{M}_{\vec{u}}$ by the definition of upper part, and $a^{(c',h')}$ will have a maximum element $\kappa(\vec{w})$ as required.

	Now we wish to add new points to $t$ and the domains of the $f$ functions, in order to ensure they contain the required co-ordinates and are ``squared off'' as specified. There are $<\kappa$-many new points needed so we can choose values for the second co-ordinates of the $f^{\theta}_{\beta}(\tau)$ that are all distinct both from pre-existing values and each other. This will avoid creating any new instances of the fifth clause of the definition of $\mathbb{Q}_{\vec{u}}$. Call the resulting condition $q=(c',h',t',f')$.

	Suppose we are given $\theta \in a^{r'}\cap\sup a^{r'}$, $\alpha,\beta \in A$, $s$ a lower part for $\mathbb{R}_{\vec{u}}$ harmonious with $c'$ past $\theta$, and $\tau,\tau'\in B$ such that $f'^{\theta}_{\alpha}(\tau) = f'^{\theta}_{\beta}(\tau)\neq f'^{\theta}_{\alpha}(\tau')=f'^{\theta}_{\beta}(\tau')$. We want
		$$s\frown((\vec{u},h'))\forces \tau\dot{\mathcal{E}}_{\alpha}\tau'\leftrightarrow\tau\dot{\mathcal{E}}_{\beta}\tau'.$$
	Note by the construction of $f'$ that we must have $\theta\in a^r\cap\sup a^r$ with $\alpha,\beta\in t^{\theta}$ and $\tau,\tau'\in\dom f^{\theta}_{\alpha}\cap \dom f^{\theta}_{\beta}$ for these equalities to be possible. We would like to split $s$ as $s_1\frown s_2$ such that $s_1$ is harmonious with $c$ past $\theta$ and $s\frown((\vec{u},h'))\leq s_1\frown((\vec{u},h))$. Unfortunately this may not be possible because the smallest triple of $s_1$ may have a second co-ordinate that includes entries from $c$. So instead we show that $s\frown((\vec{u},h'))$ forces the required statement by a density argument.

	Given any $s^*\frown((\vec{u},h''))\leq s\frown((\vec{u},h'))$ split $s^*$ as $s^*_1\frown s^*_2$ such that $\kappa(\max s^*_1)<\ssup a^r$ and $\kappa(\min s^*_2)\geq\ssup a^r$. By Lemma \ref{openHarmony} we have that $s^*_1\frown s^*_2$ is harmonious with $c'$ past $\theta$, and so also $s^*_1$ is harmonious with $c$ past $\theta$. By the conditionhood of $r$ this gives
	$$s^*_1\frown((\vec{u},h))\forces \tau\dot{\mathcal{E}}_{\alpha}\tau'\leftrightarrow\tau\dot{\mathcal{E}}_{\beta}\tau'.$$
	Strengthen $s^*_2$ to $s^{**}_2$ by shrinking the second co-ordinate of $\min s^*_2$ as necessary to ensure that it lies above $\sup a^r$, so that $s^{**}_2$ can be added below $h$. Defining $s^{**}:=s^*_1\frown s^{**}_2$ this gives us $s^{**}\frown((\vec{u},h''))\leq s^*_1\frown((\vec{u},h))$, so $s^{**}\frown((\vec{u},h''))$ forces the desideratum. We also have $s^{**}\frown((\vec{u},h''))\leq s^*\frown((\vec{u},h''))$ so we are done.
\end{proof}

Being able to perform the argument above is a reason for the second co-ordinate in the definition of the $f^{\eta}_{\alpha}$. From another perspective the second co-ordinate gives us a greater degree of flexibility in our embeddings into the jointly-universal graphs.

\subsection{Properties of $\mathbb{Q}_{\vec{u}}$}

We now prove properties of the $\mathbb{Q}_{\vec{u}}$-forcings that will be valuable when we come to iterate them. First we recall some definitions.

\begin{definition}
	A subset $X$ of a forcing $\mathbb{P}$ is {\em centred} if every finite subset of $X$ has a lower bound. The forcing $\mathbb{P}$ is {\em $\kappa$-compact} if every centred subset of size less than $\kappa$ has a lower bound.
\end{definition}

Note that $\kappa$-compactness implies $\kappa$-directed closure.

\begin{lemma} \label{Qcompact}
	The forcing $\mathbb{Q}_{\vec{u}}$ is $\kappa$-compact.
\end{lemma}

\begin{proof}
	We are given some $X\subseteq\mathbb{Q}_{\vec{u}}$ with $|X|<\kappa$. For each finite subset $x$ of $X$, take a lower bound $(c^x,h^x,t^x,f^x)$ for $x$. It is clear that we can take some $h^*$ that is below $h^x$ for all such $x$ (using the $\kappa$-completeness of $\mathcal{F}_{\vec{u}}$). Also  form $c^*$, $t^*$ and $f^*$ by unions of all the individual $c$, $t$ and $f$ from conditions in $X$. Note we do not use the $c^x$, $t^x$ and $f^x$ as these are not guaranteed to be compatible.

	We can see that $(c^*,h^*,t^*,f^*)$ satisfies the requirements for being a member of $\mathbb{Q}_{\vec{u}}$ except possibly the fifth one. Suppose we are given $\eta,\alpha,\beta,\zeta$ and $\zeta'$ together with $s$ harmonious with $c^*$ past $\eta$ such that $f^{*,\eta}_{\alpha}(\zeta) = f^{*,\eta}_{\beta}(\zeta)\neq f^{*,\eta}_{\alpha}(\zeta')=f^{*,\eta}_{\alpha}(\zeta')$. Then choose a finite set $x \subseteq X$ that contains conditions $(c,h,t,f)$ which between them witness all of the following properties:
	\begin{itemize}
		\item $s$ is harmonious with $c$ past $\eta$.
		\item $(\eta,\alpha)\in t$ and $\zeta\in\dom f^{\eta}_{\alpha}$.
		\item $(\eta,\alpha)\in t$ and $\zeta'\in\dom f^{\eta}_{\alpha}$.
		\item $(\eta,\beta)\in t$ and $\zeta\in\dom f^{\eta}_{\beta}$.
		\item $(\eta,\beta)\in t$ and $\zeta'\in\dom f^{\eta}_{\beta}$.
	\end{itemize}
	Now $(c^x,h^x,t^x,f^x)$ is a condition in $\mathbb{Q}_{\vec{u}}$ so
	$$s\frown((\vec{u},h^x))\forces\zeta\dot{\mathcal{E}}_{\alpha}\zeta'\leftrightarrow\zeta\dot{\mathcal{E}}_{\beta}\zeta'$$
	and we ensured $h^*\leq h^x$ so $s\frown((\vec{u},h^*))$ will force the same thing.
\end{proof}

\begin{definition}
	The forcing $\mathbb{P}$ has the {\em strong $\kappa^+$-chain condition} if for every sequence $\langle p_{\alpha}\mid \alpha<\kappa^+\rangle$ from $\mathbb{P}$ there are a club $C\subseteq\kappa^+$ and a regressive function $f:(C\cap\cof\kappa)\rightarrow\kappa^+$ such that for all $\alpha,\beta\in C\cap\cof\kappa$ if $f(\alpha)=f(\beta)$ then $p_{\alpha}$ and $p_{\beta}$ are compatible.
\end{definition}

Note that by Fodor's theorem this property immediately implies the usual $\kappa^+$-chain condition.

\begin{lemma} \label{Qstrong}
	The forcing $\mathbb{Q}_{\vec{u}}$ has the strong $\kappa^+$-chain condition.
\end{lemma}

\begin{proof}
	We are given $\langle p_i\mid i<\kappa^+\rangle$ a sequence of conditions from $\mathbb{Q}_{\vec{u}}$. Define $(c^i,h^i,t^i,f^i):=p^i$, $\langle f^{i,\eta}_{\alpha}\mid(\eta,\alpha)\in t^i\rangle:=f^i$ and $a^i:=a^{p^i}$. We use our ability to extend conditions as in lemma \ref{squareOff} to assume there are $A^i$ and $B^i$ such that $t^i = (a^i\cap\sup a^i)\times A^i$ and for all $(\eta,\alpha)\in t^i$ that $\dom f^{i,\eta}_{\alpha}=B^i$.

	For each $i<\kappa^+$, $\alpha\in A^i$, $\zeta,\zeta'\in B^i$ and lower part $s$ we have ``$\zeta \dot{\mathcal{E}}_{\alpha}\zeta'$'' a binary name, so by Lemma \ref{nameAnalysis} we can take $h'\leq h^i$ and a $\mathbb{R}_{\max s}\times\mathbb{B}(\max s, \gamma)$-name $\dot{y}^i_{\alpha,s}(\zeta,\zeta')$ for some $\gamma<\kappa$ with $s\frown((\vec{u},h'))\forces \zeta \dot{\mathcal{E}}_{\alpha}\zeta'\leftrightarrow \dot{y}^i_{\alpha,s}(\zeta,\zeta')$. For each $i$ we will use the $\kappa$-closure of upper parts to assume, by shrinking as necessary, that $s\frown((\vec{u},h^i))$ forces this for all such $\alpha$, $\zeta$ and $\zeta'$ and for all lower parts $s$ with $\kappa(\max s)\leq\sup a^i$ and the third co-ordinate of $\max s$ equal to zero.

	Enumerate $\bigcup_{i<\kappa^+}A^i\subseteq\kappa^{+3}$ as $\{\beta(j) \mid j<\kappa^+\}$, and for each $i<\kappa^+$ enumerate $R^i:=\{j<\kappa^+\mid \beta(j)\in A^i\}$ in increasing order as $\{j^i_{\epsilon}\mid\epsilon<\mu^i\}$ for some $\mu^i<\kappa$. Fix $\{ t(k)\mid k<\kappa^+\}$ an enumeration of points in $T$ (the tree from which the branches $x_{\alpha}$ come) and define $T^i$ to be the set of $k<\kappa^+$ such that $f^{i,\eta}_{\alpha}(\zeta)=(t(k),\nu)$ for some $\eta$, $\alpha$, $\zeta$ and $\nu$. Construct functions as follows:
	\begin{itemize}
		\item $F_1(i)=(c^i,\mu^i, R^i\cap i, B^i\cap i,T^i\cap i)$.
		\item $F_2(i)$ is the set of tuples $(\eta, \epsilon,\zeta, k,\nu)$ such that $\eta\in a^i\cap\sup a^i$, $\epsilon<\mu^i$, $\zeta<i$, $k<i$ and $f^{i,\eta}_{\beta(j^i_{\epsilon})}(\zeta)=(t(k),\nu)$.
		\item $F_3(i)$ is the set of tuples $(\epsilon, s, \zeta,\zeta', \dot{y}^i_{\beta(j^i_{\epsilon}),s}(\zeta,\zeta'))$ for $\epsilon<\mu^i$, $s$ a lower part of the form described above, and $\zeta,\zeta'\in B^i\cap i$.
	\end{itemize}
	Then we define $F(i)=(F_1(i),F_2(i),F_3(i))$. Note that $F(i)$ will be a member of
		$$(V_{\kappa}^2\times([i]^{<\kappa})^3)\times(\kappa^2\times i^2\times\kappa)^{<\kappa}\times(\kappa\times V_{\kappa}\times i^2\times V_{\kappa})^{<\kappa}.$$
	Fix an injection $G$ from
		$$(V_{\kappa}^2\times([\kappa^+]^{<\kappa})^3)\times(\kappa^2\times (\kappa^+)^2\times\kappa)^{<\kappa}\times(\kappa\times V_{\kappa}\times (\kappa^+)^2\times V_{\kappa})^{<\kappa}$$
	to $\kappa^+$. We have $\kappa^{<\kappa}=\kappa$ which implies $(\kappa^+)^{<\kappa}=\kappa^+$ so we can find a club $C_0$ such that for all points $i$ in $C_0\cap\cof\kappa$,
		$$G``((V_{\kappa}^2\times([i]^{<\kappa})^3)\times(\kappa^2\times i^2\times\kappa)^{<\kappa}\times(\kappa\times V_{\kappa}\times i^2\times V_{\kappa})^{<\kappa})\subseteq i.$$
	This will make $G\circ F:\kappa^+\rightarrow\kappa^+$ regressive on $C_0$.

	Define $C_1$ to be the club subset of $\kappa^+$ consisting of points $i'$ such that for all $i<i'$ we have that $R^i, B^i, T^i\subseteq i'$ and for all $\alpha\neq\beta$ in $A^i$ that $x_{\alpha}\upharpoonright i'\neq x_{\beta}\upharpoonright i'$.

	We will prove that the regressive function $G\circ F$ and club $C_0\cap C_1$ together serve as witnesses to the strong $\kappa^+$-chain condition. So given $i<i'$ in $C_0\cap C_1$ such that $G(F(i))=G(F(i'))$ we wish to show that $p^i$ is compatible with $p^{i'}$. Note that the properties of $C_1$ plus the fact that $F_1(i)=F_1(i')$ mean that $R^i\cap R^{i'}$, $R^i-R^{i'}$ and $R^{i'}-R^i$ are positioned in increasing order as subsets of $\kappa^+$, and likewise for $B^i$ and $T^i$.

	First consider any $\eta\in a^i\cap\sup a^i=a^{i'}\cap\sup a^{i'}$, $\alpha\in A^i\cap A^{i'}$ and $\zeta\in B^i\cap B^{i'}$. It is clear that the first co-ordinates of $f^{i,\eta}_{\alpha}(\zeta)$ and $f^{i',\eta}_{\alpha}(\zeta)$ agree, since they are just $x_{\alpha}\upharpoonright\zeta$; say this is equal to $t(k)$ and then as $k\in T^i$ and $i'\in C_1$ we get $k<i'$. Let $\alpha=:\beta(j)$, with $j\in R^i\cap R^{i'}$ so by the increasing enumeration, $j=:j^i_{\epsilon}=j^{i'}_{\epsilon}$ for some $\epsilon<\mu^i=\mu^{i'}$. We have $\zeta\in B^i\subseteq i'$. Thus the tuple $(\eta,\epsilon,\zeta,k,\pi_2(f^{i',\eta}_{\beta(j^{i'}_{\epsilon})}(\zeta)))$ will be a member of $F_2(i')$ and so also of $F_2(i)$ and we have $f^{i,\eta}_{\alpha}(\zeta)=f^{i',\eta}_{\alpha}(\zeta)$.

	The preceding argument allows us to define a putative lower bound $p^*=(c^*,h^*,t^*,f^*)$ for $p^i$ and $p^{i'}$ given by $c^*=c^i=c^{i'}$, $h^*$ any upper part below $h^i$ and $h^{i'}$, $t^*=t^i\cup t^{i'}$, and $f^{*,\eta}_{\alpha}$ equal to either $f^{i,\eta}_{\alpha}\cup f^{i',\eta}_{\alpha}$, $f^{i,\eta}_{\alpha}$ or $f^{i',\eta}_{\alpha}$ depending on whether $\alpha$ is in $A^i\cap A^{i'}$, $A^i-A^{i'}$ or $A^{i'}-A^i$ respectively. It is clear that this $p^*$ will satisfy the first four clauses of the definition of $\mathbb{Q}_{\vec{u}}$ so it remains to show the fifth.

	Define $a^*:=a^i=a^{i'}$.	We will be given $\eta\in a^*\cap\sup a^*$, $\alpha,\beta\in t^{*,\eta}$, $s$ harmonious with $c^*$ past $\eta$ and $\zeta,\zeta'\in \dom f^{*,\eta}_{\alpha}\cap f^{*,\eta}_{\beta}$ such that $f^{*,\eta}_{\alpha}(\zeta) = f^{*,\eta}_{\beta}(\zeta)\neq f^{*,\eta}_{\alpha}(\zeta')=f^{*,\eta}_{\beta}(\zeta')$. We wish to show that
	$$s\frown((\vec{u},h^*))\forces \zeta\dot{\mathcal{E}}_{\alpha}\zeta'\leftrightarrow\zeta\dot{\mathcal{E}}_{\beta}\zeta'.$$
	We see that
	$$(\alpha,\zeta),(\alpha,\zeta'),(\beta,\zeta),(\beta,\zeta')\in (A^i\times B^i)\cup(A^{i'}\times B^{i'}),$$
	which compels that either all the co-ordinates occur in a single one of $A^i\times B^i$ or $A^{i'}\times B^{i'}$, from which the result is obvious, or (without loss of generality) that we have one of the following two cases.

	\begin{case1}
	$\alpha,\beta \in A^i\cap A^{i'}$, $\zeta\in B^i-B^{i'}$ and $\zeta'\in B^{i'}-B^i$.
	\end{case1}

	We may assume $\alpha\neq\beta$. The definition of $C_1$ and the fact that $\alpha,\beta \in A^i$ ensures that $x_{\alpha}\upharpoonright i'\neq x_{\beta}\upharpoonright i'$. But $B^{i'}\cap i'=B^i\cap i$ so we must have have $\zeta'\geq i'$, giving $x_{\alpha}\upharpoonright \zeta'\neq x_{\beta}\upharpoonright \zeta'$. This contradicts $f^{*,\eta}_{\alpha}(\zeta')=f^{*,\eta}_{\beta}(\zeta')$.

	\begin{case2}
	$\alpha\in A^i-A^{i'}$, $\beta\in A^{i'}-A^i$ and $\zeta,\zeta'\in B^i\cap B^{i'}$.
	\end{case2}

	Take $j$ such that $\beta(j)=\alpha$, $j'$ such that $\beta(j')=\beta$, $\epsilon$ such that $j^i_{\epsilon}=j$ and $\epsilon'$ such that $j^{i'}_{\epsilon'}=j'$. Take $k$ such that $x_{\alpha}\upharpoonright\zeta=x_{\beta}\upharpoonright\zeta=t(k)$ and $k'$ such that $x_{\alpha}\upharpoonright\zeta'=x_{\beta}\upharpoonright\zeta=t(k')$; note that $k,k'\in T^i\cap T^{i'}\subseteq i$. Likewise $\zeta,\zeta'< i$. Combining all this information tells us that the tuples $(\eta, \epsilon, \zeta, k, \nu)$ and $(\eta, \epsilon, \zeta', k', \nu')$ appear in $F_2(i)$ for some $\nu$ and $\nu'$. Hence they also appear in $F_2(i')$ and we have
	$$f^{i',\eta}_{\beta(j^{i'}_{\epsilon})}(\zeta)=f^{i,\eta}_{\beta(j^i_{\epsilon})}(\zeta)=f^{*,\eta}_{\alpha}(\zeta)=f^{*,\eta}_{\beta}(\zeta)=f^{i',\eta}_{\beta}(\zeta)$$
	and similarly for $\zeta'$. These equalities occur entirely inside $p^{i'}$ so we can invoke its conditionhood to get
	$$s\frown((\vec{u},h^{i'}))\forces \zeta\dot{\mathcal{E}}_{\beta(j^{i'}_{\epsilon})}\zeta'\leftrightarrow\zeta\dot{\mathcal{E}}_{\beta}\zeta'.$$

	Define $\tilde{s}$ to be equal to $s$ except that the third co-ordinate of $\max \tilde{s}$ should be trivial. We will have $(\epsilon, \tilde{s}, \zeta,\zeta', \dot{y}^i_{\beta(j^i_{\epsilon}),\tilde{s}}(\zeta,\zeta'))$ in $F_3(i)$ and thus in $F_3(i')$, with $\dot{y}^i_{\alpha,\tilde{s}}(\zeta,\zeta')=\dot{y}^i_{\beta(j^i_{\epsilon}),\tilde{s}}(\zeta,\zeta'))=\dot{y}^{i'}_{\beta(j^{i'}_{\epsilon}),\tilde{s}}(\zeta,\zeta'))$. Since $s$ is below $\tilde{s}$ we know
	$$s\frown((\vec{u},h^i))\forces\zeta\dot{\mathcal{E}}_{\alpha}\zeta'\leftrightarrow\dot{y}^i_{\alpha,\tilde{s}}(\zeta,\zeta')$$
	and
	$$s\frown((\vec{u},h^{i'}))\forces\zeta\dot{\mathcal{E}}_{\beta(j^{i'}_{\epsilon})}\zeta'\leftrightarrow\dot{y}^{i'}_{\beta(j^{i'}_{\epsilon}),\tilde{s}}(\zeta,\zeta').$$
	Putting all these results together yields what we want.
\end{proof}

\section{Construction of the model} \label{construction_of_model}

We now perform an iteration of length $\kappa^{+4}$ of preparatory forcings, under the following assumptions. Note that the behaviour of the power-set function given here can be obtained from any model in which $\kappa$ is supercompact whilst preserving supercompactness.

\begin{setting}
	Let $\kappa$ be supercompact, $2^{\kappa}=\kappa^+$, $2^{\kappa^+}=\kappa^{+3}$, $2^{\kappa^{+3}}=\kappa^{+4}$ and $\lambda<\kappa$ regular uncountable.
\end{setting}

\subsection{The forcing construction}

Fix $\langle x_{\epsilon}\mid \epsilon<\kappa^{+3}\rangle$ an enumeration of the branches of the complete binary tree $T$ on $\kappa^+$.

\begin{lemma} \label{propertiesOfP}
	Let $\mathbb{P}$ be a $<\kappa$-support iteration of length $\kappa^{+4}$ of forcings that are either trivial or of the form $\mathbb{Q}_{\vec{u}}$ for some $\vec{u}$. Then $\mathbb{P}$ is $\kappa$-directed closed and has the $\kappa^+$-chain condition. Also $2^{\kappa}=2^{\kappa^+}=\kappa^{+3}$ at intermediate stages, and $2^{\kappa}=\kappa^{+4}$ at the end of the iteration.
\end{lemma}

\begin{proof}
	We have from lemma \ref{Qcompact} that the $\mathbb{Q}_{\vec{u}}$-forcings are $\kappa$-compact, hence $\kappa$-directed closed. It is clear that a $<\kappa$-support iteration of such forcings will remain $\kappa$-directed closed.

	A forcing is said to be {\em countably parallel closed} if any two descending $\omega$-sequences in it that are pointwise compatible have a common lower bound. It is clear that this property follows from $\kappa$-compactness. We also have that the component forcings are $\kappa$-closed and have the strong $\kappa^+$-chain condition, so we can invoke \cite[Theorem 1.2]{5author} to deduce that $\mathbb{P}$ has the strong $\kappa^+$-chain condition, and hence the usual $\kappa^+$-cc.

	Call the intermediate stages of the forcing $\mathbb{P}_{\gamma}$ for $\gamma<\kappa^{+4}$. We can prove by induction on $\gamma$ that $|\mathbb{P}_{\gamma}|=\kappa^{+3}$ and $(2^{\kappa^+})^{V^{\mathbb{P}_{\gamma}}}=(2^{\kappa^+}\times\kappa^{+3})^{\kappa}=\kappa^{+3}$. The latter follows from the former by the usual analysis of names together with the $\kappa^+$-cc. Conversely, conditions from $\dot{\mathbb{Q}}_{\vec{u}}$ for $\vec{u}\in V^{\mathbb{P}_{\gamma}}$ are members of $(V_{\kappa}\times 2^{\kappa}\times [\kappa\times\kappa^{+3}]^{<\kappa}\times [\kappa\times\kappa^{+3}\times\kappa]^{<\kappa})^{V^{\mathbb{P}_{\gamma}}}$, where we drop the first co-ordinate of the $f^{\eta}_{\epsilon}(\zeta)$ since it can be deduced from $\zeta$ and $\epsilon$. Thus we can use the $\kappa^+$-cc of $\mathbb{P}_{\gamma}$ and the fact that $(2^{\kappa})^{V^{\mathbb{P}_{\gamma}}}=\kappa^{+3}$ to encode them as member of $\kappa^{+3}$. Hence $|\mathbb{P}_{\gamma+1}|=\kappa^{+3}$ and the induction proceeds. Limit stages for $\gamma<\kappa^{+4}$ are immediate by the $<\kappa$-support, and then at the end we get $2^{\kappa}=\kappa^{+4}$ are desired.
\end{proof}

Define $\mathbb{L}$ to be the Laver preparatory forcing to make $\kappa$ indestructible under $\kappa$-directed closed forcing, as given in \cite{Laver}. After this forcing we still have $2^{\kappa^{+3}}=\kappa^{+4}$ so by a result from \cite{ShelahDiamond} we have a $\lozenge_{\kappa^{+4}}(\kappa^{+4}\cap\cof(\kappa^{++}))$-sequence $\langle S_{\gamma}\mid\gamma<\kappa^{+4}\rangle$. We will perform an iteration $\mathbb{P}$ of the type described above
but before doing so we wish to establish a list in $V^{\mathbb{L}}$ of all possible $\mathbb{P}$-names for subsets of $\kappa$, regardless of the sequence of $\vec{u}_{\gamma}$ we end up using to construct $\mathbb{P}$. We can do so by inductively building a list of possible $\mathbb{P}_{\gamma}$-names for subsets of $\kappa$:
\begin{itemize}
	\item For $\gamma=\delta+1$ a $\mathbb{P}_{\gamma}$-name for a subset of $\kappa$ is a $\mathbb{P}_{\delta}$-name for a $\dot{\mathbb{Q}}_{\delta}$-name for a subset of $\kappa$. Such a $\dot{\mathbb{Q}}_{\delta}$-name is, by the $\kappa^+$-cc, a function from $\kappa$ to $\dot{\mathbb{Q}}_{\delta}\times 2$ and as in the proof of \ref{propertiesOfP} members of $\dot{\mathbb{Q}}_{\delta}$ can be encoded as members of $(2^{\kappa})^{V^{\mathbb{L}*\mathbb{P}_{\delta}}}$. We note that this encoding can be done merely be looking at the shape of possible conditions, without knowledge of $\vec{u}_{\delta}$. The list of possible $\mathbb{P}_{\delta}$-names for subsets of $\kappa$ can now be used to list all the possible $\mathbb{P}_{\gamma}$-names for subsets of $\kappa$.
	\item For $\gamma$ limit the listing is straightforward because of the $\kappa^+$-cc.
\end{itemize}
Members of $\mathcal{U}^{V^{\mathbb{P}*\mathbb{L}}}$ are essentially subsets of $(2^{\kappa})^{V^{\mathbb{P}*\mathbb{L}}}$ so our listing allows us to translate between subsets of $\kappa^{+4}$ in $V^{\mathbb{L}}$ and anything that could possibly turn out to be a $\mathbb{P}$-name for a member of $\mathcal{U}$.

We are now ready to define the $<\kappa$-support iteration $\mathbb{P}=\langle \mathbb{P}_{\gamma},\mathbb{Q}_{\delta}\mid \gamma\leq\kappa^{+4},\delta<\kappa^{+4}\rangle$. At stage $\gamma$, apply the translation just established to $S_{\gamma}\subseteq\kappa^{+4}$. If the result is a $\mathbb{P}$-name for a member of $\mathcal{U}$ that is in fact already a $\mathbb{P}_{\gamma}$-name then instantiate this name in $\mathbb{P}_{\gamma}$ and call the result $\vec{u}^{\gamma}$. Use \ref{propertiesOfP} to fix $\langle \dot{\mathcal{E}}^{\gamma}_{\epsilon}\mid \epsilon<\kappa^{+3}\rangle$ an enumeration of the $\mathbb{R}_{\vec{u}^{\gamma}}$-names for graphs on $\kappa^+$. Define $\mathbb{Q}_{\gamma}=\mathbb{Q}_{\vec{u}^{\gamma}}$, working with respect to the sequences $\langle x_{\epsilon}\mid \epsilon<\kappa^{+3}\rangle$ and $\langle \dot{\mathcal{E}}^{\gamma}_{\epsilon}\mid \epsilon<\kappa^{+3}\rangle$. Otherwise take $\mathbb{Q}_{\gamma}$ to be the trivial forcing.

Let $G*H$ be $\mathbb{L}*\mathbb{P}$-generic. If $\mathbb{Q}_{\gamma}$ is non-trivial then $H(\gamma)$ will add a potential upper part which we call $h^{\gamma}$, and a sequence of functions which we call $F^{\gamma}=\langle F^{\gamma,\eta}_{\alpha}\mid \eta =\kappa(\vec{w}), \vec{w}\in\dom h^{\gamma}, \alpha<\kappa^{+3}\rangle$.

\begin{lemma} \label{staty1}
	Let $\vec{u}\in\mathcal{U}^{V[G][H]}$. Then in $V[G][H]$ there is a stationary set of $\gamma<\kappa^{+4}$ of cofinality $\kappa^{++}$ such that $\vec{u}^{\gamma}$ is the restriction of $\vec{u}$ to $V[G][H\upharpoonright \gamma]$ and $\mathbb{Q}_\gamma=\mathbb{Q}_{\vec{u}^{\gamma}}$.
\end{lemma}

\begin{proof}
	There is a club of points $\gamma$ in $\kappa^{+4}$ where the members of $(2^{\kappa})^{V[G][H]}$ listed as above by ordinals below $\gamma$ are exactly $\bigcup_{\delta<\gamma}(2^{\kappa})^{V[G][H\upharpoonright\delta]}$. For such $\gamma$ of cofinality at least $\kappa^+$ the $\kappa^+$-cc of $\mathbb{P}_{\gamma}$ makes this equal to $(2^{\kappa})^{V[G][H\upharpoonright\gamma]}$. Take a $\mathbb{P}$-name for $\vec{u}$ and use the above translation to convert it into a subset of $\kappa^{+4}$; the diamond sequence then gives us a stationary set of $\gamma<\kappa^{+4}$ of cofinality $\kappa^{++}$ such that $\vec{u}^{\gamma}$ is given by restricting $\vec{u}$ to subsets of $\kappa$ that belong to $V[G][H\upharpoonright\gamma]$. Now all the properties in the definition of $\mathcal{U}$ are $\Pi^1_2$ over $V_{\kappa}$, so there a club of $\gamma$ where the restriction of $\vec{u}$ to $V[G][H\upharpoonright\gamma]$ is a member of $\mathcal{U}^{V[G][H\upharpoonright\gamma]}$. Combining these two facts gives a stationary set of $\gamma$ where $\vec{u}$ restricts to $\vec{u}^{\gamma}$ and $\mathbb{Q}_\gamma=\mathbb{Q}_{\vec{u}^{\gamma}}$.
\end{proof}

Observe that by the properties of the Laver preparation and the fact that $\mathbb{P}$ is $\kappa$-directed closed (by lemma \ref{propertiesOfP}) we can take $j:V\rightarrow M$ witnessing that $\kappa$ is highly supercompact and $j(\mathbb{L})(\kappa)=\mathbb{P}$, and then find a master condition allowing us to extend $j$ to an embedding $j:V[G]\rightarrow M[G][H][I]$ where $I$ is generic for a highly closed forcing. We can then use the methods of section 2 to derive $\vec{u}\in\mathcal{U}^{V[G][H][I]}$ from $j$, and observe by the closure that in fact $\vec{u}\in V[G][H]$. It will then be possible to apply the above lemma to $\vec{u}$. However we will actually need to be more careful than this in the construction of our master condition, because we want to ensure that $h^{\gamma}\in \mathcal{F}_{\vec{u}}$ stationarily-often.

\begin{lemma}
	There is $\vec{u}\in\mathcal{U}^{V[G][H]}$ such that in $V[G][H]$ there is a stationary set of $\gamma<\kappa^{+4}$ of cofinality $\kappa^{++}$ such that $\vec{u}^{\gamma}$ is the restriction of $\vec{u}$ to $V[G][H\upharpoonright \gamma]$, $\mathbb{Q}_\gamma=\mathbb{Q}_{\vec{u}^{\gamma}}$, and $h^{\gamma} \in \mathcal{F}_{\vec{u}}$.
\end{lemma}

\begin{proof}
	Take $\mu$ large and $j:V\rightarrow M$ witnessing that $\kappa$ is $\mu$-supercompact with $j(\mathbb{L})(\kappa)=\mathbb{P}$ and $j(\mathbb{L})(\alpha)$ trivial for $\alpha\in(\kappa,\mu)$. We have $j$ fixing $G$ pointwise so we can extend $j$ to $j:V[G]\rightarrow M[G][H][I]$ where $I$ is some $j(\mathbb{L})/(\mathbb{L}*\mathbb{P})$-generic over $M$. We will now build a master condition in $j(\mathbb{P})$ by inductively defining a descending sequence $p_{\gamma}\in j(\mathbb{P}_{\gamma})$ for $\gamma<\kappa^{+4}$ such that $\forces p_{\gamma}\leq j``(H\upharpoonright\gamma)$.

	For $\gamma$ limit take $p_{\gamma}$ to be any lower bound of $\langle p_{\delta}\mid \delta<\gamma\rangle$, using that the forcing is highly closed. We will have $p_{\gamma+1}:=p_{\gamma}\frown(q_{\gamma})$ for $q_{\gamma}$ to be defined. We can force below $p_{\gamma}$ to lift $j$ to $j:V[G][H\upharpoonright\gamma]\rightarrow M[G][H][I][j(H\upharpoonright\gamma)]$. If $\mathbb{Q_{\gamma}}$ is the trivial forcing then so is $j(\mathbb{Q}_{\gamma})$ and we take $q_{\gamma}$ to be its unique member. Otherwise we set $q_{\gamma}=(\tilde{c}^{\gamma},\tilde{h}^{\gamma},\tilde{t}^{\gamma},\tilde{f}^{\gamma})$ with definitions as follows.
	$$\dom \tilde{c}^{\gamma} :=\dom h^{\gamma}\cup \left\{\vec{w} \in \mathcal{U}^{V[G][H]} \middle| \begin{array}{l}
		\exists i<\lambda: \vec{w}\cap V[G][H\upharpoonright\gamma]=\vec{u}^{\gamma}\upharpoonright i, h^{\gamma}_{<i}\in \mathcal{F}_{\vec{w}},\\
		\forall(c,h,t,f)\in H(\gamma): \vec{w}\in\dom j(h), \\
		\bigwedge_{(c,h,t,f)\in H(\gamma)}j(h)(\vec{w})\neq 0
	\end{array}\right\},$$
	with $\tilde{c}^{\gamma}(\vec{w}):=h^{\gamma}(\vec{w})$ for $\vec{w}\in\dom h^{\gamma}$, and
	$$\tilde{c}^{\gamma}(\vec{w}):=\bigwedge_{(c,h,t,f)\in H(\gamma)} j(h)(\vec{w})$$
	for $\vec{w}\in\dom\tilde{c}^{\gamma}-\dom h^{\gamma}$. We set
	$$\tilde{h}^{\gamma}:=\bigwedge_{(c,h,t,f)\in H(\gamma)} j(h),$$
	$\tilde{t}^{\gamma}:=(a^{(\tilde{c}^{\gamma},\tilde{h}^{\gamma})}\cap\kappa) \times j``\kappa^{+3}$, and $\dom (\tilde{f}^{\gamma})^{\eta}_{j(\alpha)}=j``\kappa^+$ and $(\tilde{f}^{\gamma})^{\eta}_{j(\alpha)}(j(\zeta))=j(F^{\gamma,\eta}_{\alpha}(\zeta))$ for all $\eta \in a^{(\tilde{c}^{\gamma},\tilde{h}^{\gamma})}\cap\kappa$, $\alpha\in\kappa^{+3}$ and $\zeta\in\kappa^+$.

	\begin{claim}
		$(\tilde{c}^{\gamma},\tilde{h}^{\gamma},\tilde{t}^{\gamma},\tilde{f}^{\gamma})\in j(\mathbb{Q}_{\gamma})$.
	\end{claim}

	\begin{proof}
		The requirement that $h^{\gamma}_{<i}\in \mathcal{F}_{\vec{w}}$ for those $\vec{w}\in \dom\tilde{c}^{\gamma}$ with $\kappa(\vec{w})=\kappa$ ensures that $\tilde{c}^{\gamma}$ is an acceptable first co-ordinate for a condition in $j(\mathbb{M}_{\vec{u}^{\gamma}})$. The first four clauses of the definition then follow from the fact that $j(\kappa)$ is large. For the fifth we are given $\eta \in a^{(\tilde{c}^{\gamma},\tilde{h}^{\gamma})}\cap\kappa$, $\alpha,\beta\in\kappa^{+3}$, $s$ a lower part for $j(\mathbb{R}_{\vec{u}^{\gamma}})$ that is harmonious with $\tilde{c}^{\gamma}$ past $\eta$, and $\zeta,\zeta'\in\kappa^+$ such that $(\tilde{f}^{\gamma})^{\eta}_{j(\alpha)}(j(\zeta))=(\tilde{f}^{\gamma})^{\eta}_{j(\beta)}(j(\zeta))\neq (\tilde{f}^{\gamma})^{\eta}_{j(\alpha)}(j(\zeta'))=(\tilde{f}^{\gamma})^{\eta}_{j(\beta)}(j(\zeta'))$. By elementarity this last assertion is equivalent to $F^{\gamma,\eta}_{\alpha}(\zeta)=F^{\gamma,\eta}_{\beta}(\zeta)\neq F^{\gamma,\eta}_{\alpha}(\zeta')=F^{\gamma,\eta}_{\beta}(\zeta')$.

		If $\kappa(\max s)<\kappa$ then use Lemma \ref{squareOff} to take a condition $(c,h,t,f)\in H(\gamma)$ with $\eta\in a^{(c,h)}\cap\sup a^{(c,h)}$, $\alpha,\beta \in t^{\eta}$, $s$ harmonious with $c$ past $\eta$, and $\zeta,\zeta'\in \dom f^{\eta}_{\alpha}\cap \dom f^{\eta}_{\beta}$. Then
		$f^{\eta}_{\alpha}(\zeta)=f^{\eta}_{\beta}(\zeta)\neq f^{\eta}_{\alpha}(\zeta')=f^{\eta}_{\beta}(\zeta')$ so we get
		$$s\frown((\vec{u}^{\gamma},h))\forces \zeta\dot{\mathcal{E}}^{\gamma}_{\alpha}\zeta' \leftrightarrow \zeta\dot{\mathcal{E}}^{\gamma}_{\beta}\zeta'.$$
		Now $s\frown((j(\vec{u}^{\gamma}),\tilde{h}^{\gamma}))\leq s\frown((j(\vec{u}^{\gamma}),j(h)))$ so together with elementarity we obtain
		$$s\frown((j(\vec{u}^{\gamma}),\tilde{h}^{\gamma})) \forces j(\zeta) j(\dot{\mathcal{E}}^{\gamma}_{\alpha})j(\zeta') \leftrightarrow j(\zeta) j(\dot{\mathcal{E}}^{\gamma}_{\beta})j(\zeta')$$
		as required.

		Otherwise we can write $s$ as $s_1\frown((\vec{w}, d, p))$ for some $\vec{w}\in\dom \tilde{c}^{\gamma}-\dom h^{\gamma}$. We will show that $s\frown((j(\vec{u}^{\gamma}),\tilde{h}^{\gamma}))$ forces what we want by a density argument. Suppose we are given an extension $s^*\frown((j(\vec{u}^{\gamma}),h^*))$; express $s^*$ as $s^*_1\frown((\vec{w}, d^*,p^*))\frown s^*_2$. Lemma \ref{openHarmony} tells us that $s^*_1\frown((\vec{w}, d^*,p^*))$ remains harmonious with $\tilde{c}^{\gamma}$ past $\eta$, so we can use Lemma \ref{squareOff} to take $(c,h,t,f)\in H(\gamma)$ with $\eta\in a^{(c,h)}\cap\sup a^{(c,h)}$, $\alpha,\beta \in t^{\eta}$, $s^*_1$ harmonious with $c$ past $\eta$, and $\zeta,\zeta'\in \dom f^{\eta}_{\alpha}\cap \dom f^{\eta}_{\beta}$. As before the conditionhood of $(c,h,t,f)$ followed by the elementarity of $j$ give that
		$$s^*_1\frown((j(\vec{u}^{\gamma}),j(h))) \forces j(\zeta) j(\dot{\mathcal{E}}^{\gamma}_{\alpha})j(\zeta') \leftrightarrow j(\zeta) j(\dot{\mathcal{E}}^{\gamma}_{\beta})j(\zeta').$$
		The harmoniousness of $s$ with $\tilde{c}^{\gamma}$ tells us that $d^*\leq d\leq h^{\gamma}\leq c\cup h$ so we can refine $s^*$ to $s^{**}$ by strengthening $d^*$ to $d^{**}\leq h =j(h)\upharpoonright\kappa$. We also have $\vec{w}\in\dom\tilde{c}^{\gamma}\subseteq\dom j(h)$ and $p^*\leq p \leq \tilde{c}^{\gamma}(\vec{w})\leq j(h)(\vec{w})$, so $(\vec{w},d^{**},p^*)$ is addable below $(j(\vec{u}^{\gamma}),j(h))$. So is  $s^*_2$ (because it is addable below $\tilde{h}^{\gamma}\leq j(h)$) yielding
		$$s^{**}\frown((j(\vec{u}^{\gamma}),h^*)) \forces j(\zeta) j(\dot{\mathcal{E}}^{\gamma}_{\alpha})j(\zeta') \leftrightarrow j(\zeta) j(\dot{\mathcal{E}}^{\gamma}_{\beta})j(\zeta').$$
		And $s^{**}\frown((j(\vec{u}^{\gamma}),h^*))$ is also below $s^*\frown((j(\vec{u}^{\gamma}),h^*))$, concluding the proof of the claim.
	\end{proof}

	It is immediate that $\forces q_{\gamma}\leq j``H(\gamma)$ so $\forces p_{\gamma}\leq j``(H\upharpoonright\gamma)$; this finishes the inductive definition. Take $p$ a lower bound of the sequence of $p_{\gamma}$ as our master condition and force below it to obtain a $j(\mathbb{P})$-generic filter. Then we can extend $j$ to $j:V[G][H]\rightarrow M[G][H][I][j(H)]$, where $j(H)$ is the filter for $j(\mathbb{P})$ just obtained. This embedding will witness a high degree of generic supercompactness so as in section 2 we can in $V[G][H][I][j(H)]$ derive an ultrafilter sequence $\vec{u}$ from it, and show $\vec{u}\in\mathcal{U}$; we also get the associated supercompact ultrafilter sequence $\vec{u}^*=\langle z, u^*_i, K^*_i \mid i<\lambda\rangle$ and the associated projection $\pi$. The $\mu$-closure of the $j(\mathbb{L})/(\mathbb{L}*\mathbb{P})*j(\mathbb{P})$-forcing gives us that $\vec{u}\in V[G][H]$. Then we can invoke Lemma \ref{staty1} to see that there are stationarily-many $\gamma<\kappa^{+4}$ where $\vec{u}$ restricts to $\vec{u}^{\gamma}$ and $\mathbb{Q}_{\gamma}=\mathbb{Q}_{\vec{u}^{\gamma}}$. We wish to show that $h^{\gamma}\in \mathcal{F}_{\vec{u}}$ for such $\gamma$ and will do so by proving by induction on $i$ that $h^{\gamma}_i\in \mathcal{F}_{\vec{u},i}$.

	Given any $(c,h,t,f)\in H(\gamma)$ we have $h_i\in \mathcal{F}_{\vec{u}^{\gamma},i}\subseteq \mathcal{F}_{\vec{u},i}$ so $\dom h_i\in u_i$, which gives
	\begin{align*}
		& \pi^{-1}` ` \dom h_i \in u^*_i \\
		\Rightarrow& \forall_{u^*_i}\vec{w}^* \pi(\vec{w}^*)\in\dom h_i \\
		\Rightarrow& j(\pi)(\vec{u}^*\upharpoonright i)\in\dom j(h_i) \\
		\Rightarrow& \vec{u}\upharpoonright i \in\dom j(h_i)
	\end{align*}
	Also by definition of $\Fil(K^*_i)$ there is an $A\in u^*_i$ with $h_i\geq b(K^*_i,A)$, from which
	\begin{align*}
		& \forall\vec{w}\in\dom h_i: h_i(\vec{w}) \geq \bigvee\{K^*_i(\vec{w}^*)\mid \vec{w}^*\in A, \pi(\vec{w}^*)=\vec{w}\}\\
		\Rightarrow& \forall\vec{w}^*\in A: h_i(\pi(\vec{w}^*)) \geq K^*_i(\vec{w}^*)\\
		\Rightarrow& j(h_i)(\vec{u}\upharpoonright i) \geq j(K^*_i)(\vec{u}^*\upharpoonright i),
	\end{align*}
	using that $A\in u^*_i$ and $j(\pi)(\vec{u}^*\upharpoonright i)=\vec{u}\upharpoonright i$.

	Therefore $j(K^*_i)(\vec{u}^*\upharpoonright i)$ witnesses that $\bigwedge_{(c,h,t,f)\in H(\gamma)}j(h)(\vec{u}\upharpoonright i)\neq 0$. By the induction hypothesis we have $h^{\gamma}_{<i}\in \mathcal{F}_{\vec{u}\upharpoonright i}$ so we established have all of the requirements necessary for $\vec{u}\upharpoonright i\in\dom\tilde{c}^{\gamma}$. We have also shown that $\tilde{c}^{\gamma}(\vec{u}\upharpoonright i)\geq j(K^*_i)(\vec{u}^*\upharpoonright i)$. Now forcing below $p_{\gamma+1}$ ensures that $\tilde{c}^{\gamma}$ is an initial segment of $j(h^{\gamma})$ so we have
	\begin{align*}
		& j(h^{\gamma})(\vec{u}^\upharpoonright i)\geq j(K^*)(\vec{u}^*\upharpoonright i)\\
		\Rightarrow& \forall\vec{w}^*\in B: h^{\gamma}(\pi(\vec{w}^*))\geq K^*(\vec{w}^*) \mbox{ for some } B\in u^*_i\\
		\Rightarrow& \forall\vec{w}\in\pi ` `B: h^{\gamma}_i(\vec{w})\geq \bigvee\{K^*_i(\vec{w}^*)\mid \vec{w}^*\in B, \pi(\vec{w}^*)=\vec{w}\}\\
		\Rightarrow& h^{\gamma}_i\geq b(K^*_i,B)
	\end{align*}
	which gives $h^{\gamma}_i\in \mathcal{F}_{\vec{u},i}$ as desired.
\end{proof}

Fix a $\vec{u}$ and $S\subseteq\kappa^{+4}$ stationary as given by this lemma. Take $J$ that is $\mathbb{R}_{\vec{u}}$-generic over $V[G][H]$, forcing below an upper part whose domain is made up of sequences of length less than $\lambda$, so that $J$ generates a generic sequence $\langle \vec{w}_{\alpha}, g_{\alpha}\mid\alpha<\lambda\rangle$ as discussed in sub-section \ref{defineR}. For any $\gamma\in S$ we observe by the characterisation of genericity in Lemma \ref{characteriseGenericity} that $J$ is geometric for $\mathbb{R}_{\vec{u}}$ and hence $\mathbb{R}_{\vec{u}^{\gamma}}$, so it is generic for $\mathbb{R}_{\vec{u}^{\gamma}}$ and we can form the extension $V[G][H\upharpoonright \gamma][J]$.

\subsection{The jointly universal family}

We now fix some $\gamma \in S$ and define a graph $\mathcal{E}^{\gamma}$ on $T\times\kappa$ that is intended to be universal with respect to the graphs in $V[G][H\upharpoonright \gamma][J]$. We have $h^{\gamma}\in \mathcal{F}_{\vec{u}}$ so start by fixing $\beta<\lambda$ such that $\lh\vec{w}_{\beta}=0$ and for all $\alpha>\beta$ we have $\vec{w}_{\alpha} \in \dom h^{\gamma}$ and $h^{\gamma}(\vec{w}_{\alpha})\in g_{\alpha}$. Define $\eta:=\kappa(\vec{w}_{\beta})$. Define $\mathcal{E}^{\gamma}_{\epsilon}$ to be the realisation of $\dot{\mathcal{E}}^{\gamma}_{\epsilon}$ in $V[G][H\upharpoonright \gamma][J]$. For $z, z' \in T\times\kappa$ we define $z \mathcal{E}^{\gamma} z'$ if there exist $\epsilon$, $\zeta$ and $\zeta'$ such that $F^{\gamma,\eta}_{\epsilon}(\zeta)=z$ and $F^{\gamma,\eta}_{\epsilon}(\zeta')=z'$ with $\zeta \mathcal{E}^{\gamma}_{\epsilon} \zeta'$ in $V[G][H\upharpoonright \gamma][J]$.

\begin{lemma}
	Let $\gamma\in S$, $\eta$ as above and $\epsilon<\kappa^{+3}$. Then in $V[G][H\upharpoonright \gamma][J]$ the function $F^{\gamma,\eta}_{\epsilon}$ is an embedding of $\mathcal{E}^{\gamma}_{\epsilon}$ into $\mathcal{E}^{\gamma}$.
\end{lemma}

\begin{proof}
	It is clear from the definition that every edge in $\mathcal{E}^{\gamma}_{\epsilon}$ is mapped to one in $\mathcal{E}^{\gamma}$, so we need to show the converse. Consider $\epsilon$, $\zeta$ and $\zeta'$ such that $F^{\gamma,\eta}_{\epsilon}(\zeta) \mathcal{E}^{\gamma} F^{\gamma,\eta}_{\epsilon}(\zeta)$. Observe first that the values of $\zeta$ and $\zeta$ are deducible from their targets under $F^{\gamma,\eta}_{\epsilon}$, so there must be some $\epsilon'$ with $\zeta \mathcal{E}^{\gamma}_{\epsilon'} \zeta'$ such that
	$$F^{\gamma,\eta}_{\epsilon}(\zeta)=F^{\gamma,\eta}_{\epsilon'}(\zeta)\neq F^{\gamma,\eta}_{\epsilon}(\zeta')=F^{\gamma,\eta}_{\epsilon'}(\zeta').$$
	Take a condition $s\frown((\vec{u}^{\gamma},h)) \in J$ such that $s\frown((\vec{u}^{\gamma},h))\forces\zeta \dot{\mathcal{E}}^{\gamma}_{\epsilon'}\zeta'$, $s$ extends past $\eta$, and $\vec{w}_{\beta}$ occurs in $s$. Use Lemma \ref{squareOff} to take a condition $(c,h',t,f)\in H(\gamma)$ such that $c$ extends past $\max s$, $a^{(c,h')}$ has a maximal element, $h'\leq h$, $(\eta,\epsilon),(\eta,\epsilon') \in t$ and $\zeta,\zeta'\in \dom f^{\eta}_{\epsilon}\cap \dom f^{\eta}_{\epsilon'}$. Our aim is to find a lower part $s'$ such that:
	\begin{enumerate} [(i)]
		\item $s'$ is harmonious with $c$ past $\eta$.
		\item $s'\frown((\vec{u}^{\gamma},h'))\leq s\frown((\vec{u}^{\gamma},h))$.
		\item $s'\frown((\vec{u}^{\gamma},h'))\in J$.
	\end{enumerate}
	Then we will use (i) to invoke the fifth clause of the definition of $\mathbb{Q}_{\vec{u}^{\gamma}}$ for $(c,h',t,f)$ to see that
	$$s'\frown((\vec{u}^{\gamma},h'))\forces \zeta \dot{\mathcal{E}}^{\gamma}_{\epsilon} \zeta' \leftrightarrow \zeta \dot{\mathcal{E}}^{\gamma}_{\epsilon'} \zeta'$$
	which by (ii) will give $s'\frown((\vec{u}^{\gamma},h'))\forces\zeta \dot{\mathcal{E}}^{\gamma}_{\epsilon}\zeta'$ and then by (iii) we will be done.

	We construct $s'$ from $s$ as follows:
	\begin{itemize}
		\item Leave triples $(\vec{w}_{\alpha},d,p)$ with $\kappa(\vec{w}_{\alpha})\leq\eta$ (i.e. $\alpha\leq\beta$) unchanged.
		\item For $(\vec{w}_{\alpha},d,p)\in s$ with $\alpha>\beta$ replace with $(\vec{w}_{\alpha},d\wedge(c\upharpoonright\kappa(\vec{w}_{\alpha})),p\wedge c(\vec{w}_{\alpha}))$.
		\item The set of $\kappa(\vec{w}_{\alpha})$ is a club, and $\ssup a^{(c,h')}$ is a successor. This means we can take $\beta'$ maximal such that $\vec{w}_{\beta'}\in \dom c$. Then add $(\vec{w}_{\beta'},(h\upharpoonright\kappa(\vec{w}_{\beta'}))\wedge(c\upharpoonright\kappa(\vec{w}_{\beta'})),h(\vec{w}_{\beta'})\wedge c(\vec{w}_{\beta'}))$ to the end of $s$.
	\end{itemize}
	For $(\vec{w}_{\alpha},d,p)\in s$ with $\alpha>\beta$ note that $c$ is an initial segment of an upper part $h^{\gamma}$ and $\vec{w}_{\alpha}\in\dom h^{\gamma}$ so we are guaranteed that $c\upharpoonright\kappa(\vec{w}_{\alpha})\in \mathcal{F}_{\vec{w}_{\alpha}}$ for such $\alpha$. Also $p\in g_{\alpha}$, and $c(\vec{w}_{\alpha})=h^{\gamma}(\vec{w}_{\alpha})\in g_{\alpha}$ by choice of $\beta$, so $p$ and $c(\vec{w}_{\alpha})$ are compatible. The same holds for $\beta'$, ensuring $s'$ is a valid lower part. Now we check that it has the required properties.
	\begin{enumerate} [(i)]
		\item This is immediate from the definition (and the reason for the appearance of $c$ in it).
		\item The new triple of $s'$ must be addable to $s\frown((\vec{u}^{\gamma},h))$ on account of its being in $J$, and we have taken care to respect $h$ here.
		\item We will use the requirements from Definition \ref{defnGenericFilter} for a condition to belong to the generic filter $J$ associated with $\langle \vec{w}_{\alpha}, g_{\alpha}\mid\alpha<\lambda\rangle$. The first clause is clear so we consider the second. For $\alpha<\beta$ we have $\vec{w}_{\alpha}$ below a triple of $s$ that is not modified, so all is well. For $\beta<\alpha<\beta'$ we have $\vec{w}_{\alpha}\in h^{\gamma}$ with $h^{\gamma}(\vec{w}_{\alpha})\in g_{\alpha}$. Now $c$ is an initial segment of $h^{\gamma}$ that extends to $\vec{w}_{\beta'}$ so $\vec{w}_{\alpha}\in \dom c$ and $c(\vec{w}_{\alpha})\in g_{\alpha}$. This means that the modifications made to the members of $s$ are unproblematic. (It is for this step that we had to add the extra triple to $s'$.) Finally for $\alpha>\beta'$ we have that $\kappa(\vec{w}_{\alpha})>\kappa(\max \dom c)$ so the fact that $(c,h',t,f)\in H(\gamma)$ and $\vec{w}_{\alpha}\in\dom h^{\gamma}$ tells us that $\vec{w}_{\alpha}\in\dom h'$; likewise $g_{\alpha}\ni h^{\gamma}(\vec{w}_{\alpha})\leq h'(\vec{w}_{\alpha})$.
	\end{enumerate}
\end{proof}

We can now conclude the proof. Take a sequence $\langle\delta_i\mid i<\kappa^{++}\rangle$ of points from $S$ such that $\delta:=\sup \delta_i$ is in $S$. Our final model will be $V[G][H\upharpoonright\delta][J]$ and the family of universal graphs will be $\{\mathcal{E}^{\delta_i}\mid i<\kappa^{++}\}$. Given some graph $\mathcal{E}$ in the model, take a $\mathbb{R}_{\vec{u}^{\delta}}$-name $\dot{\mathcal{E}}$ in $V[G][H\upharpoonright\delta]$ for it. By the $\kappa^+$-cc of $\mathbb{R}_{\vec{u}^{\delta}}$ this name can be coded as a subset of $\kappa^+$ and then by the $\kappa^+$-cc of the forcing iteration we can find some $i<\kappa^{++}$ such that $\dot{\mathcal{E}}$ is in $V[G][H\upharpoonright\delta_i]$. Since $\vec{u}^{\delta_i}$ is a restriction of $\vec{u}^{\delta}$ we see that $\mathbb{R}_{\vec{u}^{\delta_i}}$ will also interpret the name as $\mathcal{E}$, and the lemma above shows that it can be embedded into $\mathcal{E}^{\delta_i}$.

By Lemma \ref{propertiesOfP} $\mathbb{L}\times\mathbb{P}_{\gamma}$ preserves all cardinals, and then by Proposition \ref{preserveCardinals} $\mathbb{R}_{\vec{u}}$ changes $\kappa$ to $\aleph_{\lambda}$ and preserves all larger cardinals. We have proved the following theorem,

\begin{theorem}
	Let $\kappa$ be supercompact and $\lambda<\kappa$ regular uncountable. Then there is a forcing extension in which $\kappa=\aleph_{\lambda}$, $2^{\aleph_{\lambda}}=2^{\aleph_{\lambda+1}}=\aleph_{\lambda+3}$ and there is a jointly universal family of graphs on $\aleph_{\lambda+1}$ of size $\aleph_{\lambda+2}$.
\end{theorem}

\section*{Acknowledgement}

Thanks go to my advisor James Cummings for all his help with this paper.

\bibliographystyle{plain}

\end{document}